\documentclass[12pt]{amsart}

\usepackage{amsmath,amssymb,latexsym,amsthm,newlfont,enumerate}

\usepackage{graphicx}
\usepackage[all]{xy}
\input diagxy

\theoremstyle{plain}

\newtheorem{thm}{Theorem}[section]

\newtheorem{pro}[thm]{Proposition}

\newtheorem{lem}[thm]{Lemma}

\newtheorem{cor}[thm]{Corollary}

\newtheorem*{propL}{Property $\mcal L_A^G$}

\theoremstyle{definition}

\newtheorem{dfn}[thm]{Definition}

\newtheorem{dfnlm}[thm]{Definition-Lemma}
\newtheorem{rem}[thm]{Remark}

\theoremstyle{remark}

\DeclareMathOperator{\mult}{mult}
\DeclareMathOperator{\ord}{ord}
\DeclareMathOperator{\res}{res}
\DeclareMathOperator{\im}{Im}
\DeclareMathOperator{\Int}{int}
\DeclareMathOperator{\relint}{relint}
\DeclareMathOperator{\Supp}{Supp}
\DeclareMathOperator{\Bs}{Bs}

\DeclareMathOperator{\Div}{Div}
\DeclareMathOperator{\WDiv}{WDiv}
\DeclareMathOperator{\ddiv}{div}

\DeclareMathOperator{\Mob}{Mob}
\DeclareMathOperator{\bMob}{\mathbf{Mob}}
\DeclareMathOperator{\Fix}{Fix}

\newcommand{\R}{\mathbb{R}}
\newcommand{\Q}{\mathbb{Q}}
\newcommand{\N}{\mathbb{N}}
\newcommand{\Z}{\mathbb{Z}}

\newcommand{\C}{\mathbb{C}}

\newcommand{\B}{\mathbf{B}}

\newcommand{\m}{\mathbf{m}}

\newcommand{\bigcone}{\mathrm{Big}}

\newcommand{\OO}{\mathcal{O}}

\newcommand{\mcal}{\mathcal}

\title[Towards finite generation without the MMP]{Towards finite generation of the canonical ring without the MMP}
\date{16 December 2008}
\author{Vladimir Lazi\'c}
\address{Department of Pure Mathematics and Mathematical Statistics, Uni\-ver\-si\-ty of Cambridge, Wilberforce Road, Cambridge CB3 0WB, UK}
\email{V.Lazic@dpmms.cam.ac.uk}

\begin{document}

\begin{abstract}
This paper is the first of two steps in a project to prove finite generation of the log canonical ring without Mori theory.
\end{abstract}

\maketitle
\bibliographystyle{amsalpha}

\tableofcontents

\section{Introduction}

In this paper I establish the first of two steps in a project to prove finite generation of the log canonical ring without the
Minimal Model Program. I prove:

\begin{thm}\label{cor:can}
Let $(X,\Delta)$ be a projective klt pair and assume Property $\mcal L_A^G$ in dimensions $\leq\dim X$. Then the log canonical ring
$R(X,K_X+\Delta)$ is finitely generated.
\end{thm}

Property $\mcal L_A^G$ is stated below.
Let me sketch the strategy for the proof of finite generation and present difficulties that arise on the way. The natural idea
is to pick a smooth divisor $S$ on $X$ and to restrict the algebra to it. If we are very lucky, the restricted algebra will be
finitely generated and we might hope that the generators lift to generators on $X$. There are several issues with this approach.

Firstly, in order to obtain something meaningful on $S$, $S$ should be a log canonical centre of some pair $(X,\Delta')$ such that
$R(X,K_X+\Delta)$ and $R(X,K_X+\Delta')$ share a common truncation. This issue did not exist in the case of pl flips.

Secondly, even if the restricted algebra were finitely generated, the same might not be obvious for the kernel of the restriction map.
Note that the ``kernel issue'' also did not exist in the case of pl flips, since the relative Picard number $=1$ ensured that the kernel
was a principal ideal, at least after shrinking the base and passing to a truncation. So far this seems to have been the greatest conceptual issue
in attempts to prove the finite generation by the plan just outlined.

Thirdly, the natural choice is to use the Hacon-M\textsuperscript{c}Kernan extension theorem, see Theorem \ref{thm:hmck} below,
and hence we must be able to ensure that $S$ does not belong to the stable base locus of $K_X+\Delta'$.

The idea to resolve the kernel issue is to view $R(X,K_X+\Delta)$ as a subalgebra of a larger algebra, which would a priori contain generators
of the kernel. In practice this means that the new algebra will have higher rank grading. Namely, we will see that the rank corresponds to
the number of components of $\Delta$ and of an effective divisor $D\sim_\Q K_X+\Delta$.

Let me illustrate this on a basic example which will model the general lines of the proof in Section \ref{proofmain}. Say we wanted
to prove by induction that the ring $R(X,H)$ was finitely generated, where $H$ is an ample divisor. By passing to a truncation and
by taking a general member of $|\kappa H|$ for $\kappa\gg0$, we may assume that $H$ is smooth and very ample. By Serre's vanishing
the restriction map $\rho_k\colon H^0(X,\OO_X(kH))\rightarrow H^0(H,\OO_H(kH))$ is surjective for all $k$, and by induction
$R(H,\OO_H(H))$ is finitely generated. If $\sigma_H\in H^0(X,\OO_X(H))$ is a section such that $\ddiv\sigma_H=H$ and $\mcal H$ is a finite
set of generators of the finite dimensional vector space $\bigoplus_{i=1}^d H^0(X,\OO_X(iH))$, for some $d$, such that the set $\{s_{|H}:s\in\mcal H\}$
generates $R(H,\OO_H(H))$,
it is easy to see that $\mcal H\cup\{\sigma_H\}$ is a set of generators of $R(X,H)$, since $\ker(\rho_k)=H^0\big(X,\OO_X((k-1)H)\big)\cdot\sigma_H$.
A version of this idea applies to the case of pl flips and forms the basis of the construction of Shokurov and Hacon-M\textsuperscript{c}Kernan.

It is natural to try and restrict to a component of $\Delta$, the issue of course being that $(X,\Delta)$ does not have log canonical centres.
Therefore I allow restrictions to components of some effective divisor $D\sim_\Q K_X+\Delta$, and a tie-breaking-like technique allows
to create log canonical centres. Algebras encountered this way are, in effect, plt algebras, and their restriction is handled in
Section \ref{plt}. This is technically the most involved part of the proof.

Since the algebras we consider are of higher rank, not all divisors will have the same log canonical centres. I therefore restrict
to available centres, and lift generators from algebras that live on different divisors. Since the restrictions will also be
algebras of higher rank, the induction process must start from them. Thus, the main technical result of the paper is the following.

\begin{thm}\label{thm:main}
Let $X$ be a smooth projective variety, and for $i=1,\dots,\ell$ let $D_i=k_i(K_X+\Delta_i+A)$, where $A$ is an ample $\Q$-divisor and
$(X,\Delta_i+A)$ is a log smooth log canonical pair with $|D_i|\neq\emptyset$. Assume Property $\mcal L_A^G$
in dimensions $\leq\dim X$. Then the Cox ring $R(X;D_1,\dots,D_\ell)$ is finitely generated.
\end{thm}

Property $\mcal L_A^G$ in the statement of Theorems \ref{cor:can} and \ref{thm:main} describes the convex geometry of the set of log canonical
pairs with big boundaries in terms of divisorial components of the stable base loci. More precisely:

\begin{propL}
Let $X$ be a smooth variety projective over an affine variety $Z$, $B$ a simple normal crossings divisor on $X$ and
$A$ a general ample $\Q$-divisor. Let $V\subset\Div(X)_\R$ be the vector space spanned by the components of $B$ and let
$\mcal L_V=\{\Theta\in V:(X,\Theta)\textrm{ is log canonical}\}$; this is a rational polytope in $V$.
Then for any component $G$ of $B$, the set
$$\mcal L_A^G=\{\Phi\in\mcal L_V:G\not\subset\B(K_X+\Phi+A)\}$$
is a rational polytope.
\end{propL}

Precise definitions are given in Section \ref{sec:2}. This property is a consequence of the MMP, see Proposition \ref{pro:2} below.
Here I would like to comment on a possible strategy to prove Property $\mcal L_A^G$ without using Mori theory. It seems likely
that the method of the proof will be quite similar to techniques used in Section 5 below to handle finite generation of
restricted algebras. The method involved is very recent, and has appeared in \cite{Hac08} in order to handle the proof of Non-vanishing
theorem; \cite{Pau08} gives a proof of that statement without the MMP.
The necessary technical tools were developed in \cite[Section 3]{Laz07}
in order to prove that certain superlinear functions are in fact piecewise linear, and they might be of use if one wishes to prove that
certain sets are polytopes.

Finally, it is my hope that the techniques of this paper could be adapted to handle finite generation in the case of log canonical singularities
and the abundance conjecture.\vspace{5mm}
\paragraph{\bf Acknowledgements}
I am indebted to my supervisor Alessio Corti for the philosophical point that sometimes higher rank is better than rank $1$ and that
starting from simple examples is essential. I would like to express my gratitude for his encouragement, support and continuous inspiration.
I am very grateful to Christopher Hacon for suggesting that methods from \cite{Hac08} might be useful in the context of finite generation of
the restricted algebra. Many thanks to S\'ebastien Boucksom, Paolo Cascini, Anne-Sophie Kaloghiros, Mihai P\u{a}un and Miles Reid for useful
comments and suggestions. I am supported by Trinity College, Cambridge.

\section{Notation and conventions}\label{sec:2}

Unless stated otherwise, varieties in this paper are normal over $\C$ and projective over an affine variety $Z$.
The group of Weil, respectively Cartier, divisors on a variety $X$ is denoted by $\WDiv(X)$, respectively $\Div(X)$.
We denote $\WDiv(X)^{\kappa\geq0}=\{D\in\WDiv(X):\kappa(X,D)\geq0\}$, and similarly for $\Div(X)^{\kappa\geq0}$, where $\kappa$
is the Iitaka dimension. Subscripts denote the rings in which the coefficients are taken.

We say an ample $\Q$-divisor $A$ on a variety $X$ is {\em (very) general\/} if there is a sufficiently divisible
positive integer $k$ such that $kA$ is very ample and $kA$ is a (very) general section of $|kA|$.
In particular we can assume that for some $k\gg0$, $kA$ is a smooth divisor on $X$.

For any two divisors $P=\sum p_iE_i$ and $Q=\sum q_iE_i$ on $X$ set
$$P\wedge Q=\sum\min\{p_i,q_i\}E_i.$$

I use basic properties of b-divisors, see \cite{Cor07}. The cone of mobile b-divisors on $X$ is denoted by $\bMob(X)$.

For the definition and basic properties of multiplier ideals used in this paper see \cite{HM08}.

The sets of non-negative (respectively non-positive) rational and real numbers are denoted by $\Q_+$ and $\R_+$ ($\Q_-$ and $\R_-$ respectively).
\vspace{5mm}
\paragraph{\bf Convex geometry}
If $\mcal{S}=\sum\N e_i$ is a submonoid of $\N^n$, I denote $\mcal{S}_\Q=\sum\Q_+e_i$ and $\mcal{S}_\R=\sum\R_+e_i$.
A monoid $\mcal{S}\subset\N^n$ is {\em saturated\/} if $\mcal{S}=\mcal{S}_\R\cap\N^n$.

If $\mcal{S}=\sum_{i=1}^n\N e_i$ and $\kappa_1,\dots,\kappa_n$ are positive integers,
the submonoid $\mcal{S}'=\sum_{i=1}^n\N \kappa_ie_i$ is called a {\em truncation\/} of $\mcal{S}$.
If $\kappa_1=\dots=\kappa_n=\kappa$, I denote $\mcal S^{(\kappa)}:=\sum_{i=1}^n\N \kappa e_i$, and this truncation does not depend on a
choice of generators of $\mcal S$.

A submonoid $\mcal{S}=\sum\N e_i$  of $\N^n$ (respectively a cone $\mcal{C}=\sum\R_+ e_i$ in $\R^n$) is called {\em simplicial\/}
if its generators $e_i$ are linearly independent in $\R^n$, and the $e_i$ form a {\em basis\/} of $\mcal{S}$
(respectively $\mcal{C}$).

I often use Gordan's lemma without explicit mention, see \cite[Lemma 2.4]{Laz07}, and also that if $\lambda\colon\mcal M\rightarrow\mcal S$
is an additive surjective map between finitely generated saturated monoids, and if $\mcal C$ is a rational polyhedral cone in $\mcal S_\R$, then
$\lambda^{-1}(\mcal S\cap\mcal C)=\mcal M\cap\lambda^{-1}(\mcal C)$. In particular, the inverse image of a saturated finitely
generated submonoid of $\mcal S$ is a saturated finitely generated submonoid of $\mcal M$.

For a polytope $\mcal P\subset\R^n$, I denote $\mcal P_\Q=\mcal P\cap\Q^n$. A polytope is {\em rational\/} if it is the convex hull of finitely
many rational points.

If $\mcal B\subset\R^n$ is a convex set, then $\R_+\mcal B$ will denote the set $\{rb:r\in\R_+,b\in\mcal B\}$. In particular,
if $\mcal B$ is a rational polytope, $\R_+\mcal B$ is a rational polyhedral cone. The dimension of the rational polytope $\mcal P$,
denoted $\dim\mcal P$, is the dimension of the smallest rational affine space containing $\mcal P$.

Let $\mcal S\subset \N^n$ be a finitely generated monoid, $\mcal{C}\in\{\mcal S,\mcal S_\Q,\mcal S_\R\}$ and $V$ an $\R$-vector space.
A function $f\colon\mcal{C}\rightarrow V$ is: {\em positively homogeneous\/}
if $f(\lambda x)=\lambda f(x)$ for $x\in\mcal C,\lambda\geq0$; {\em superadditive\/} if $f(x)+f(y)\leq f(x+y)$ for $x,y\in\mcal{C}$;
{\em $\Q$-superadditive\/} if $\lambda f(x)+\mu f(y)\leq f(\lambda x+\mu y)$ for $x,y\in\mcal C,\lambda,\mu\in\Q_+$; {\em $\Q$-additive\/} if
the previous inequality is an equality; and {\em superlinear\/} if $\lambda f(x)+\mu f(y)\leq f(\lambda x+\mu y)$ for
$x,y\in\mcal S_\R,\lambda,\mu\in\R_+$. Similarly for {\em additive, subadditive, sublinear\/}.
It is {\em piecewise additive\/} if there is a finite polyhedral decomposition $\mcal{C}=\bigcup\mcal{C}_i$ such that
$f_{|\mcal{C}_i}$ is additive for every $i$; additionally, if each $\mcal C_i$ is a rational cone, it is {\em rationally piecewise additive\/}.
Similarly for (rationally) piecewise linear.
Assume furthermore that $f$ is linear on $\mcal{C}$ and $\dim\mcal{C}=n$. The {\em linear extension of $f$ to $\R^n$\/} is the unique linear function
$\ell\colon\R^n\rightarrow V$ such that $\ell_{|\mcal{C}}=f$.

In this paper the {\em relative interior\/} of a cone $\mcal{C}=\sum\R_+e_i\subset\R^n$, denoted by $\relint\mcal{C}$, is the topological interior
of $\mcal{C}$ in the space $\sum\R e_i$ union the origin. If $\dim\mcal{C}=n$, we instead call it the {\em interior\/} of $\mcal{C}$ and
denote it by $\Int\mcal{C}$. The boundary of a closed set $\mcal C$ is denoted by $\partial\mcal C$.\vspace{5mm}
\paragraph{\bf Asymptotic invariants}
The standard references on asymptotic invariants arising from linear series are \cite{Nak04,ELMNP}.

\begin{dfn}
Let $X$ be a variety and $D\in\WDiv(X)_\R$. For $k\in\{\Z,\Q,\R\}$, define
$$|D|_k=\{C\in\WDiv(X)_k:C\geq0,C\sim_kD\}.$$
If $T$ is a prime divisor on $X$ such that $T\not\subset\Fix|D|$, then $|D|_T$ denotes the image of the linear system $|D|$ under restriction
to $T$. The {\em stable base locus\/} of $D$ is $\B(D)=\bigcap_{C\in|D|_\R}\Supp C$ if $|D|_\R\neq\emptyset$, otherwise we define $\B(D)=X$.
The {\em diminished base locus\/} is $\B_-(D)=\bigcup_{\varepsilon>0}\B(D+\varepsilon A)$ for an ample divisor $A$; this definition does not depend on
a choice of $A$. In particular $\B_-(D)\subset\B(D)$.
\end{dfn}

It is elementary that $\B(D_1+D_2)\subset\B(D_1)\cup\B(D_2)$ for $D_1,D_2\in\WDiv(X)_\R$. In other words, the set
$\{D\in\WDiv(X)_\R:x\notin\B(D)\}$ is convex for every point $x\in X$.
By \cite[Lemma 3.5.3]{BCHM}, $\B(D)=\bigcap_{C\in|D|_\Q}\Supp C$ when $D$ is a $\Q$-divisor, which is the standard definition of the stable base locus.

\begin{dfn}
Let $Z$ be a closed subvariety of a smooth variety $X$ and let $D$ be an effective $\Q$-divisor on $X$.
The {\em asymptotic order of vanishing of $D$ along $Z$\/} is
$$\ord_Z\|D\|=\inf\{\mult_ZC:C\in|D|_\Q\}.$$
\end{dfn}

More generally, one can consider any discrete valuation $\nu$ of $k(X)$ and define
$$\nu\|D\|=\inf\{\nu(C):C\in|D|_\Q\}$$
for an effective $\Q$-divisor $D$. Then \cite{ELMNP} shows that $\nu\|D\|=\nu\|E\|$ if $D$ and $E$ are numerically equivalent big divisors, and that
$\nu$ extends to a sublinear function on $\bigcone(X)_\R$.

\begin{rem}\label{rem:10}
When $X$ is projective, Nakayama in \cite{Nak04} defines a function $\sigma_Z\colon\overline{\bigcone(X)}\rightarrow\R_+$ by
$$\sigma_Z(D)=\lim_{\varepsilon\downarrow0}\ord_Z\|D+\varepsilon A\|$$
for any ample $\R$-divisor $A$, and shows that it agrees with $\ord_Z\|\cdot\|$ on big classes. Analytic properties of these
invariants were studied in \cite{Bou04}.
\end{rem}

We can define the restricted version of the invariant introduced.

\begin{dfn}
Let $S$ be a smooth divisor on a smooth variety $X$ and let $D\in\Div(X)_\Q^{\kappa\geq0}$ be such that $S\not\subset\B(D)$.
Let $P$ be a closed subvariety of $S$. The {\em restricted asymptotic order of vanishing of $|D|_S$ along $P$\/} is
$$\ord_P\|D\|_S=\inf\{\mult_P C:kC\in|kD|_S\textrm{ for some }k\geq1\}.$$
\end{dfn}

\begin{rem}\label{rem:11}
Similarly as in Remark \ref{rem:10}, \cite{Hac08} introduces a function
$\sigma_P\|\cdot\|_S\colon\mcal C_-\rightarrow\R_+$ by
$$\sigma_P\|D\|_S=\lim_{\varepsilon\downarrow0}\ord_P\|D+\varepsilon A\|_S$$
for any ample $\R$-divisor $A$, where $\mcal C_-\subset\overline{\bigcone(X)}$ is the set of classes of divisors $D$ such that $S\not\subset\B_-(D)$.
Then one can define a formal sum $N_\sigma\|D\|_S=\sum\sigma_P\|D\|_S\cdot P$ over all prime divisors $P$ on $S$. If $S\not\subset\B(D)$, then
for every $\varepsilon_0>0$ we have $\lim_{\varepsilon\downarrow\varepsilon_0}\sigma_P\|D+\varepsilon A\|_S=\ord_P\|D+\varepsilon_0A\|_S$
for any ample divisor $A$ on $X$ similarly as in \cite[Lemma 2.1.1]{Nak04}, cf.\ \cite[Lemma 7.8]{Hac08}.
\end{rem}

In this paper I need a few basic properties cf.\ \cite[Lemma 7.14]{Hac08}.

\begin{lem}\label{lem:restrictedord}
Let $S$ be a smooth divisor on a smooth projective variety $X$, let $D\in\Div(X)_\Q^{\kappa\geq0}$ be such that $S\not\subset\B(D)$
and let $P$ be a closed subvariety of $S$. If $A$ is an ample $\Q$-divisor on $X$, then $\ord_P\|D+A\|_S\leq\ord_P\|D\|_S$, and in particular
$\sigma_P\|D\|_S\leq\ord_P\|D\|_S$. If $\sigma_P\|D\|_S=0$, then there is a positive integer $l$ such that $\mult_P\Fix|l(D+A)|_S=0$.
\end{lem}
\begin{proof}
The first statement is trivial. For the second one, we have $\ord_P\|D+\frac12A\|_S=0$. Set $n=\dim X$, let $H$ be a very ample divisor on $X$
and fix a positive integer $l$ such that $H'=\frac l2A-(K_X+S)-(n+1)H$ is very ample. Let $\Delta\sim_\Q D+\frac12A$ be a $\Q$-divisor
such that $S\not\subset\Supp\Delta$ and $\mult_P\Delta_{|S}<1/l$. We have
$$H^i(X,\mcal J_{l\Delta_{|S}}(K_S+H'_{|S}+(n+1)H_{|S}+l\Delta_{|S}+mH_{|S}))=0$$
for $m\geq-n$ by Nadel vanishing. Since $l(D+A)\sim_\Q K_X+S+H'+(n+1)H+l\Delta$, the sheaf $\mcal J_{l\Delta_{|S}}(l(D+A))$ is globally generated
by \cite[Lemma 5.7]{HM08} and its sections lift to $H^0(X,l(D+A))$ by \cite[Lemma 4.4(3)]{HM08}. Since $\mult_P(l\Delta_{|S})<1$,
$\mcal J_{l\Delta_{|S}}$ does not vanish along $P$ and so $\mult_P\Fix|l(D+A)|_S=0$.
\end{proof}

\section{Higher rank algebras}\label{sec:3}

In this section I adapt some of the definitions from \cite{Laz07} to suit the context of this paper.

\begin{dfn}\label{dfn:1}
Let $X$ be a variety, $\mcal{S}$ a finitely generated submonoid of $\N^r$, and let
$\mu\colon\mcal{S}\rightarrow\WDiv(X)^{\kappa\geq0}$ be an additive map.
The algebra
$$R(X,\mu(\mcal{S}))=\bigoplus_{s\in\mcal{S}}H^0(X,\OO_X(\mu(s)))$$
is called the {\em divisorial $\mcal S$-graded algebra associated to $\mu$.\/} When $\mcal S=\bigoplus_{i=1}^\ell\N e_i$ is a simplicial cone,
the algebra $R(X,\mu(\mcal S))$ is called the {\em Cox ring associated to $\mu$\/}, and is denoted also by $R(X;\mu(e_1),\dots,\mu(e_\ell))$.
\end{dfn}

\begin{rem}\label{rem:1}
Algebras considered in this paper are {\em algebras of sections\/}. I will occasionally, and without explicit mention,
view them as algebras of rational functions, in particular to be able to write $H^0(X,D)\simeq H^0(X,\Mob(D))\subset k(X)$.

Assume now that $X$ is smooth, $D\in\Div(X)$ and that $\Gamma$ is a prime divisor
on $X$. If $\sigma_\Gamma$ is the global section of $\OO_X(\Gamma)$ such that $\ddiv \sigma_\Gamma=\Gamma$, from the exact sequence
$$0\rightarrow H^0(X,\OO_X(D-\Gamma))\stackrel{\cdot \sigma_\Gamma}{\longrightarrow}H^0(X,\OO_X(D))\stackrel{\rho_{D,\Gamma}}{\longrightarrow}
H^0(\Gamma,\OO_\Gamma(D))$$
we define $\res_\Gamma H^0(X,\OO_X(D))=\im(\rho_{D,\Gamma})$. For $\sigma\in H^0(X,\OO_X(D))$, I denote
$\sigma_{|\Gamma}:=\rho_{D,\Gamma}(\sigma)$. Observe that
\begin{equation}\label{eq:1}
\ker(\rho_{D,\Gamma})=H^0(X,\OO_X(D-\Gamma))\cdot\sigma_\Gamma,
\end{equation}
and that $\res_\Gamma H^0(X,\OO_X(D))=0$ if $\Gamma\subset\Bs|D|$. If $D\sim D'$ such that the restriction $D'_{|\Gamma}$ is defined,
then
$$\res_\Gamma H^0(X,\OO_X(D))\simeq\res_\Gamma H^0(X,\OO_X(D'))\subset H^0(\Gamma,\OO_\Gamma(D'_{|\Gamma})).$$
The {\em restriction of $R(X,\mu(\mcal S))$ to $\Gamma$\/} is defined as
$$\res_\Gamma R(X,\mu(\mcal S))=\bigoplus_{s\in\mcal{S}}\res_\Gamma H^0(X,\OO_X(\mu(s))).$$
This is an $\mcal S$-graded, not necessarily divisorial algebra.
\end{rem}

\begin{rem}
Under assumptions from Definition \ref{dfn:1} we define the map $\bMob_\mu\colon\mcal S\rightarrow\bMob(X)$
by $\bMob_\mu(s)=\bMob(\mu(s))$ for every $s\in\mcal S$. Then we have a b-divisorial algebra
$$R(X,\bMob_\mu(\mcal S))\simeq R(X,\mu(\mcal S))$$
as defined in \cite{Laz07}. If $\mcal S'$ is a finitely generated submonoid of $\mcal S$, I use $R(X,\mu(\mcal S'))$
to denote $R(X,\mu_{|\mcal S'}(\mcal S'))$. If $\mcal S$ is a submonoid of $\WDiv(X)^{\kappa\geq0}$ and $\iota\colon\mcal S\rightarrow\mcal S$
is the identity map, I use $R(X,\mcal S)$ to denote $R(X,\iota(\mcal S))$.
\end{rem}

The following lemma summarises the basic properties of higher rank finite generation.

\begin{lem}\label{lem:1}
Let $\mcal{S}\subset\N^n$ be a finitely generated monoid and let $R=\bigoplus_{s\in\mcal S}R_s$ be an $\mcal S$-graded algebra.
\begin{enumerate}
\item Let $\mcal S'$ be a truncation of $\mcal S$. If the $\mcal S'$-graded algebra $R'=\bigoplus_{s\in\mcal S'}R_s$
is finitely generated over $R_0$, then $R$ is finitely generated over $R_0$.

\item Assume furthermore that $\mcal S$ is saturated and let $\mcal S''\subset\mcal S$ be a finitely generated saturated submonoid.
If $R$ is finitely generated over $R_0$,
then the $\mcal S''$-graded algebra $R''=\bigoplus_{s\in\mcal S''}R_s$ is finitely generated over $R_0$.

\item Let $X$ be a variety and let
$\mu\colon\mcal{S}\rightarrow\WDiv(X)^{\kappa\geq0}$ be an additive map. If there exists a rational polyhedral subdivision
$\mcal{S}_\R=\bigcup_{i=1}^k\Delta_i$ such that, for each $i$, $\bMob_{\mu|\Delta_i\cap\mcal{S}}$ is an additive map up to truncation, then
the algebra $R(X,\mu(\mcal{S}))$ is finitely generated.
\end{enumerate}
\end{lem}
\begin{proof}
See \cite[Lemmas 5.1 and 5.2]{Laz07} and \cite[Lemma 4.8]{ELMNP}.
\end{proof}

I will need the following result in the proof of Proposition \ref{pro:1} and in Section \ref{plt}.

\begin{lem}\label{concave}
Let $X$ be a variety, $\mcal{S}\subset\N^r$ a finitely generated monoid and let $f\colon\mcal{S}\rightarrow G$ be
a superadditive map to a monoid $G$ which is a subset
of $\WDiv(X)$ or $\bMob(X)$, such that for every $s\in\mcal{S}$ there is a positive integer $\iota_s$ such that $f_{|\N\iota_ss}$ is an additive map.

Then there is a unique $\Q$-superadditive function $f^\sharp\colon\mcal{S}_\Q\rightarrow G_\Q$ such that
for every $s\in\mcal{S}$ there is a positive integer $\lambda_s$ with $f(\lambda_s s)=f^\sharp(\lambda_s s)$.
Furthermore, let $\mcal{C}$ be a rational polyhedral subcone of $\mcal{S}_\R$. Then $f_{|\mcal{C}\cap\mcal{S}}$ is additive up to
truncation if and only if $f^\sharp_{|\mcal{C}\cap\mcal S_\Q}$ is $\Q$-additive.

If $\mu\colon\mcal S\rightarrow\Div(X)$ is an additive map and $\m=\bMob_\mu$ is such that for every $s\in\mcal{S}$ there is a positive integer $\iota_s$
such that $\m_{|\N\iota_ss}$ is an additive map, then we have
\begin{equation}\label{eq:2}
\m^\sharp(s)=\overline{\mu(s)}-\sum\big(\ord_E\|\mu(s)\|\big)E,
\end{equation}
where the sum runs over all geometric valuations $E$ on $X$.
\end{lem}
\begin{proof}
See the proof of \cite[Lemma 5.4]{Laz07}. Equation \eqref{eq:2} is a restatement of the definition of $\m^\sharp$ from that proof.
\end{proof}

\begin{dfn}
In the context of Lemma \ref{concave}, the function $f^\sharp$ is called {\em the straightening of $f$\/}.
\end{dfn}

\begin{pro}\label{pro:1}
Let $X$ be a variety, $\mcal S\subset\N^r$ a finitely generated saturated monoid and $\mu\colon\mcal S\rightarrow \WDiv(X)^{\kappa\geq0}$ an
additive map. Let $\mcal L$ be a finitely generated submonoid of $\mcal S$ and assume $R(X,\mu(\mcal S))$ is finitely generated. Then
$R(X,\mu(\mcal L))$ is finitely generated. Moreover, the map $\m=\bMob_{\mu|\mcal L}$ is piecewise additive up to truncation.
In particular, there is a positive integer $p$ such that $\bMob_\mu(ips)=i\bMob_\mu(ps)$ for every $i\in\N$ and every $s\in\mcal L$.
\end{pro}
\begin{proof}
Denote $\mcal M=\mcal L_\R\cap\N^r$. By Lemma \ref{lem:1}(2), $R(X,\mu(\mcal M))$
is finitely generated, and by the proof of \cite[Theorem 4.1]{ELMNP}, there is a finite rational polyhedral subdivision
$\mcal M_\R=\bigcup\Delta_i$ such that for every geometric valuation $E$ on $X$, the map $\ord_E\|\cdot\|$ is $\Q$-additive on
$\Delta_i\cap\mcal M_\Q$ for every $i$. Since for every saturated rank $1$ submonoid $\mcal R\subset\mcal M$ the algebra $R(X,\mu(\mcal R))$ is
finitely generated by Lemma \ref{lem:1}(2), the map $\m_{\mcal R\cap\mcal L}$ is additive up to truncation by \cite[Lemma 2.3.53]{Cor07}
and thus there is the well-defined straightening $\m^\sharp\colon \mcal L_\Q\rightarrow\bMob(X)_\Q$ since $\mcal M_\Q=\mcal L_\Q$.
Then \eqref{eq:2} implies that the map $\m^\sharp|_{\Delta_i\cap\mcal L_\Q}$ is $\Q$-additive for every $i$, hence by Lemma \ref{concave}
the map $\m$ is piecewise additive up to truncation, and therefore $R(X,\mu(\mcal L))$ is finitely generated by Lemma \ref{lem:1}(3).
\end{proof}

The following lemma shows that finite generation implies certain boundedness on the convex geometry of boundaries.

\begin{lem}\label{bounded}
Let $(X,\Delta=B+A)$ be a log smooth klt pair, where $A$ is a general ample $\Q$-divisor, $B$ is an effective $\R$-divisor, and assume
that no component of $B$ is in $\B(K_X+\Delta)$. Assume Property $\mcal L_A^G$ and Theorem \ref{thm:main} in dimension $\dim X$.
Let $V\subset\Div(X)_\R$ be the vector space spanned by the components of $B$ and $W\subset V$ the smallest rational affine subspace
containing $B$. Then there is a constant $\eta>0$ and a positive integer $r$ such that if $\Phi\in W$ and
$k$ is a positive integer such that $\|\Phi-B\|<\eta$ and $k(K_X+\Phi+A)/r$ is Cartier, then no component of $B$ is in $\Fix|k(K_X+\Phi+A)|$.
\end{lem}
\begin{proof}
Let $K_X$ be a divisor such that $\OO_X(K_X)\simeq\omega_X$ and $\Supp A\not\subset\Supp K_X$, and let $\Lambda\subset\Div(X)$ be the monoid
spanned by components of $K_X,B$ and $A$.
Let $G$ be a components of $B$. By Property $\mcal L_A^G$ there is a rational polytope $\mcal P\subset W$ such that $\Delta\in\relint\mcal P$
and $G\not\subset\B(K_X+\Phi+A)$ for every $\Phi\in\mcal P$.
Let $D_1,\dots,D_\ell$ be generators of $\R_+(K_X+A+\mcal P)\cap\Lambda$.
By Theorem \ref{thm:main} the Cox ring $R(X;D_1,\dots,D_\ell)$ is finitely generated, and thus so is the algebra $R(X,\Lambda)$ by projection.
By Proposition \ref{pro:1} there is a rational polyhedral cone $\mcal C\subset\Lambda_\R$ such that $\Delta\in\mcal C$ and the map
$\bMob_{\iota|\mcal C\cap\Lambda^{(r)}}$ is additive for some positive integer $r$, where $\iota\colon\Lambda\rightarrow\Lambda$ is the identity map.
In particular, if $\Phi\in\mcal C\cap\mcal P$ and $k(K_X+\Phi+A)/r$ is Cartier, then $G\not\subset\Fix|k(K_X+\Phi+A)|$. Pick $\eta$ such that
$\Phi\in\mcal C\cap\mcal P$ whenever $\Phi\in W$ and $\|\Phi-\Delta\|<\eta$. We can take $\eta$ and $r$ to work for all components of $B$,
and we are done.
\end{proof}

To conclude this section, I show how results of \cite{BCHM} imply Property $\mcal L_A^G$. Of course, a hope is that this will be proved without
Mori theory.

\begin{pro}\label{pro:2}
Property $\mcal L_A^G$ follows from the MMP.
\end{pro}
\begin{proof}
Let $K_X$ be a divisor such that $\OO_X(K_X)\simeq\omega_X$ and $\Supp A\not\subset\Supp K_X$, and let $\Lambda$ be the monoid in $\Div(X)$ generated
by the components of $K_X,B$ and $A$. Let $\iota\colon\Lambda\rightarrow\Lambda$ be the identity map, and denote
$\mcal S=\R_+(K_X+A+\mcal L_V)\cap\Lambda$. Since $\mcal L_V$ is a rational polytope, $\mcal S$ is a
finitely generated monoid and let $D_i$ be
generators of $\mcal S$. By \cite[Corollary 1.1.9]{BCHM}, the Cox ring $R(X;D_1,\dots,D_k)$ is finitely generated, thus so is the algebra
$R(X,\mcal S)$ by projection. The set $\mcal M=\{D\in\mcal S:|D|_\Q\neq\emptyset\}$ is a convex cone, and therefore finitely
generated since $R(X,\mcal S)$ is finitely generated, so I can assume $\mcal M=\mcal S$.
By Proposition \ref{pro:1}, the map $\bMob_\iota$ is piecewise additive up to truncation,
which proves that the closure $\mcal C$ of the set $(\mcal L_A^G)_\Q$ is a rational polytope, and I claim it equals $\mcal L_A^G$. Otherwise there exists
$\Phi\in\mcal L_A^G\backslash\mcal C$, and therefore the convex hull of the set $\mcal C\cup\{\Phi\}$, which is by convexity a subset
of $\mcal L_A^G$, contains a rational point $\Phi'\in\mcal L_A^G\backslash\mcal C$, a contradiction.
\end{proof}

\section{Diophantine approximation}

I need a few results from Diophantine approximation theory.

\begin{lem}\label{diophant2}
Let $\Lambda\subset\R^n$ be a lattice spanned by rational vectors, and let $V=\Lambda\otimes_\Z\R$. Fix a vector $v\in V$ and denote $X=\N v+\Lambda$.
Then the closure of $X$ is symmetric with respect to the origin. Moreover,
if $\pi\colon V\rightarrow V/\Lambda$ is the quotient map, then the closure of $\pi(X)$ is a finite disjoint union of connected components.
If $v$ is not contained in any proper rational affine subspace of $V$, then $X$ is dense in $V$.
\end{lem}
\begin{proof}
Let $G$ be the closure of $\pi(X)$. Then $G$ is a closed infinite subgroup of the compact group $V/\Lambda$.
The connected component $G_0$ of the identity in $G$ is a Lie subgroup of $V/\Lambda$ and so by \cite[Theorem 15.2]{Bum04},
$G_0$ is a torus. Thus $G_0=V_0/\Lambda_0$, where $V_0=\Lambda_0\otimes_\Z\R$ is a rational subspace of $V$.
Since $G/G_0$ is discrete and compact, it is finite, and it is straightforward that $X$ is symmetric with respect to the origin.
Therefore a translate of $v$ by a rational vector is contained in $V_0$,
and so if $v$ is not contained in any proper rational affine subspace of $V$, then $V_0=V$.
\end{proof}

The next result is \cite[Lemma 3.7.7]{BCHM}.

\begin{lem}\label{lem:surround}
Let $x\in\R^n$ and let $W$ be the smallest rational affine space containing $x$. Fix a positive integer $k$ and a positive real number $\varepsilon$.
Then there are $w_1,\dots,w_p\in W\cap\Q^n$ and positive integers $k_1,\dots,k_p$ divisible by $k$, such that $x=\sum_{i=1}^p r_iw_i$ with $r_i>0$
and $\sum r_i=1$, $\|x-w_i\|<\varepsilon/k_i$ and $k_iw_i/k$ is integral for every $i$.
\end{lem}

I will need a refinement of this lemma when the smallest rational affine space containing a point is not necessarily of maximal dimension.

\begin{lem}\label{lem:approximation}
Let $x\in\R^n$, let $0<\varepsilon,\eta\ll1$ be rational numbers and let $w_1\in\Q^n$ and $k_1\in\N$ be such that $\|x-w_1\|<\varepsilon/k_1$ and
$k_1w_1$ is integral. Then there are $w_2,\dots,w_m\in\Q^n$, positive integers $k_2,\dots,k_m$ such that
$\|x-w_i\|<\varepsilon/k_i$ and $k_iw_i$ is integral for every $i$, and positive numbers $r_1,\dots,r_m$ such that $x=\sum_{i=1}^m r_iw_i$
and $\sum r_i=1$. Furthermore, we can assume that $w_3,\dots,w_m$ belong to the smallest rational affine space containing $x$, and we can write
$$x=\frac{k_1}{k_1+k_2}w_1+\frac{k_2}{k_1+k_2}w_2+\xi,$$
with $\|\xi\|<\eta/(k_1+k_2)$.
\end{lem}
\begin{proof}
Let $W$ be the minimal rational affine subspace containing $x$, let $\pi\colon\R^n\rightarrow\R^n/\Z^n$ be the quotient map
and let $G$ be the closure of the set $\pi(\N x+\Z^n)$. Then by Lemma \ref{diophant2} we have $\pi(-k_1x)\in G$ and
there is $k_2\in\N$ such that $\pi(k_2x)$ is in the connected component of $\pi(-k_1x)$ in $G$ and
$\|k_2x-y\|<\eta$ for some $y\in\R^n$ with $\pi(y)=\pi(-k_1x)$. Thus there is a point $w_2\in\Q^n$ such that
$k_2w_2\in\Z^n$, $\|k_2x-k_2w_2\|<\varepsilon$ and the open segment $(w_1,w_2)$ intersects $W$.

Pick $t\in(0,1)$ such that $w_t=tw_1+(1-t)w_2\in W$, and choose, by Lemma \ref{lem:surround}, rational points $w_3,\dots,w_m\in W$ and positive
integers $k_3,\dots,k_m$ such that $k_iw_i\in\Z^n$, $\|x-w_i\|<\varepsilon/k_i$ and $x=\sum_{i=3}^mr_iw_i+r_tw_t$ with $r_t>0$ and all $r_i>0$,
and $r_t+\sum_{i=3}^m r_i=1$. Thus $x=\sum_{i=1}^mr_iw_i$ with $r_1=tr_t$ and $r_2=(1-t)r_t$.

Finally, observe that the vector $y/k_2-w_2$ is parallel to the vector $x-w_1$ and $\|y-k_2w_2\|=\|k_1x-k_1w_1\|$. Denote $z=x-y/k_2$. Then
$$\frac{x-w_1}{(w_2+z)-x}=\frac{x-w_1}{w_2-y/k_2}=\frac{k_2}{k_1},$$
so
$$x=\frac{k_1}{k_1+k_2}w_1+\frac{k_2}{k_1+k_2}(w_2+z)=\frac{k_1}{k_1+k_2}w_1+\frac{k_2}{k_1+k_2}w_2+\xi,$$
where $\|\xi\|=\|k_2z/(k_1+k_2)\|<\eta/(k_1+k_2)$.
\end{proof}

\begin{rem}\label{rem:2}
Assuming notation from the previous proof, the connected components of $G$ are precisely the connected components of the set $\pi(\bigcup_{k>0}kW)$.
Therefore $y/k_2\in W$.
\end{rem}

\begin{rem}\label{rem:3}
Assume $\lambda\colon V\rightarrow W$ is a linear map between vector spaces such that $\lambda(V_\Q)\subset W_\Q$.
Let $x\in V$ and let $H\subset V$ be the smallest rational affine subspace
containing $x$. Then $\lambda(H)$ is the smallest rational affine subspace of $W$ containing $\lambda(x)$. Otherwise, assume $H'\neq\lambda(H)$ is
the smallest rational affine subspace containing $\lambda(x)$. Then $\lambda^{-1}(H')$ is a rational affine subspace containing $x$ and
$H\not\subset\lambda^{-1}(H')$, a contradiction.
\end{rem}

\section{Restricting plt algebras}\label{plt}

In this section I establish one of the technically most difficult steps in the proof of Theorem \ref{thm:main}. Crucial results
and techniques will be those used to prove Non-vanishing theorem in \cite{Hac08} using methods developed in \cite{HM08}, and the techniques
of \cite[Section 3]{Laz07}.

The key result is the following Hacon-M\textsuperscript{c}Kernan extension theorem \cite[Theorem 6.2]{HM08}, whose proof relies
on deep techniques initiated by \cite{Siu98}.

\begin{thm}\label{thm:hmck}
Let $\pi\colon X\rightarrow Z$ be a projective morphism to a normal affine variety $Z$, where $(X,\Delta=S+A+B)$ is a purely log terminal pair,
$S=\lfloor\Delta\rfloor$ is irreducible, $(X,S)$ is log smooth, $A$ is a general ample $\Q$-divisor and $(S,\Omega+A_{|S})$ is canonical,
where $\Omega=(\Delta-S)_{|S}$. Assume $S\not\subset\B(K_X+\Delta)$, and let
$$F=\liminf_{m\rightarrow\infty}\textstyle\frac1m\Fix|m(K_X+\Delta)|_S.$$
If $\varepsilon>0$ is any rational number such that $\varepsilon(K_X+\Delta)+A$ is ample and if $\Phi$ is any $\Q$-divisor on $S$
and $k>0$ is any integer such that both $k\Delta$ and $k\Phi$ are Cartier, and $\Omega\wedge(1-\frac{\varepsilon}{k})F\leq\Phi\leq\Omega$, then
$$|k(K_S+\Omega-\Phi)|+k\Phi\subset|k(K_X+\Delta)|_S.$$
\end{thm}

The immediate consequence is:

\begin{cor}\label{cor:hmck}
Let $\pi\colon X\rightarrow Z$ be a projective morphism to a normal affine variety $Z$, where $(X,\Delta=S+A+B)$ is a purely log terminal pair,
$S=\lfloor\Delta\rfloor$ is irreducible, $(X,S)$ is log smooth, $A$ is a general ample $\Q$-divisor and $(S,\Omega+A_{|S})$ is canonical,
where $\Omega=(\Delta-S)_{|S}$. Assume $S\not\subset\B(K_X+\Delta)$, and let $\Phi_m=\Omega\wedge\frac1m\Fix|m(K_X+\Delta)|_S$ for every $m$
such that $m\Delta$ is Cartier. Then
$$|m(K_S+\Omega-\Phi_m)|+m\Phi_m=|m(K_X+\Delta)|_S.$$
\end{cor}

\begin{dfnlm}
Let $(X,\Delta)$ be a log pair and let $f\colon Y\rightarrow X$ be a proper birational morphism. We can write uniquely
$$K_Y+B_Y=f^*(K_X+\Delta)+E_Y,$$
where $B_Y$ and $E_Y$ are effective with no common components and $E_Y$ is $f$-exceptional. There is a well-defined {\em boundary\/}
b-divisor $\B(X,\Delta)$ given by $\B(X,\Delta)_Y=B_Y$ for every model $Y\rightarrow X$.
\end{dfnlm}
\begin{proof}
Let $h\colon Y'\rightarrow Y$ be a log resolution and denote $g=f\circ h$. Pushing forward
$K_{Y'}+B_{Y'}=g^*(K_X+\Delta)+E_{Y'}$ via $h_*$ yields
$$K_Y+h_*B_{Y'}=f^*(K_X+\Delta)+h_*E_{Y'},$$
and thus $h_*B_{Y'}=B_Y$ since $h_*B_{Y'}$ and $h_*E_{Y'}$ have no common components.
\end{proof}
\begin{lem}\label{disjoint}
Let $(X,\Delta)$ be a log canonical pair. There exists a log resolution $Y\rightarrow X$ such that the components of $\{\B(X,\Delta)_Y\}$ are disjoint.
\end{lem}
\begin{proof}
See \cite[Proposition 2.36]{KM98} or \cite[Lemma 6.7]{HM05}.
\end{proof}

The main result of this section is the following.

\begin{thm}\label{thm:2}
Let $X$ be a smooth variety, $S$ a smooth prime divisor and $A$ a very general ample $\Q$-divisor on $X$.
For $i=1,\dots,\ell$ let $D_i=k_i(K_X+\Delta_i)$, where $(X,\Delta_i=S+B_i+A)$ is a log smooth plt pair with $\lfloor \Delta_i\rfloor=S$
and $|D_i|\neq\emptyset$.
Assume Property $\mcal L_A^G$ in dimensions $\leq\dim X$ and Theorem \ref{thm:main} in dimension $\dim X-1$.
Then the algebra $\res_S R(X;D_1,\dots,D_\ell)$ is finitely generated.
\end{thm}
\begin{proof}
{\em Step 1.\/}
I first show that we can assume $S\notin\Fix|D_i|$ for all $i$.

To prove this, let $K_X$ be a divisor with $\OO_X(K_X)\simeq\omega_X$ and
$\Supp A\not\subset\Supp K_X$, and let $\Lambda$ be the monoid in $\Div(X)$ generated
by the components of $K_X$ and all $\Delta_i$.
Denote $\mcal C_S=\{P\in\Lambda_\R:S\notin\B(P)\}$. By Property $\mcal L_A^G$, the set
$\mcal A=\sum_i\R_+D_i\cap\mcal C_S$ is a rational polyhedral cone.

The monoid $\sum_{i=1}^\ell\R_+D_i\cap\Lambda$ is finitely generated and let $P_1,\dots,P_q$ be its generators with $P_i=D_i$ for $i=1,\dots,\ell$.
Let $\mu\colon\bigoplus_{i=1}^q\N e_i\rightarrow\Div(X)$ be an additive map from a simplicial monoid such that $\mu(e_i)=P_i$.
Therefore $\mcal S=\mu^{-1}(\mcal A\cap\Lambda)\cap\bigoplus_{i=1}^\ell\N e_i$ is a finitely generated monoid and let $h_1,\dots,h_m$ be generators of
$\mcal S$, and observe that $\mu(h_i)$ is a multiple of an adjoint bundle for every $i$.

Since $\res_S H^0(X,\mu(s))=0$ for every $s\in\big(\bigoplus_{i=1}^\ell\N e_i\big)\backslash\mcal S$, we have that the algebra
$\res_S R(X,\mu(\bigoplus_{i=1}^\ell\N e_i))=\res_S R(X;D_1,\dots,D_\ell)$ is finitely generated if and only if
$\res_S R(X,\mu(\mcal S))$ is. Since we have the diagram
$$\bfig
 \square/->`>`>`>/<1300,500>[R(X;\mu(h_1),\dots,\mu(h_m))`R(X,\mu(\mcal S))
 `\res_S R(X;\mu(h_1),\dots,\mu(h_m))`\res_S R(X,\mu(\mcal S));```]
 \efig$$
where the horizontal maps are natural projections and the vertical maps are restrictions
to $S$, it is enough to prove that the restricted algebra $\res_S R(X;\mu(h_1),\dots,\mu(h_m))$ is finitely generated. By passing to a truncation,
I can assume further that $S\notin\Fix|\mu(h_i)|$ for $i=1,\dots,m$.\\[2mm]
{\em Step 2.\/}
Therefore I can assume $\mcal S=\bigoplus_{i=1}^\ell\N e_i$ and $\mu(e_i)=D_i$ for every $i$. For $s=\sum_{i=1}^\ell t_ie_i\in\mcal S_\Q$ and
$t_s=\sum_{i=1}^\ell t_ik_i$, denote $\Delta_s=\sum_{i=1}^\ell t_ik_i\Delta_i/t_s$ and $\Omega_s=(\Delta_s-S)_{|S}$.
Observe that
$$R(X;D_1,\dots,D_\ell)=\bigoplus_{s\in\mcal S}H^0(X,t_s(K_X+\Delta_s)).$$
In this step I show that we can assume that $(S,\Omega_s+A_{|S})$ is terminal for every $s\in\mcal S_\Q$.

Let $\sum F_k=\bigcup_i\Supp B_i$, and denote $\B_i=\B(X,\Delta_i)$ and $\B=\B(X,S+\nu\sum_kF_k+A)$,
where $\nu=\max_{i,k}\{\mult_{F_k}B_i\}$. By Lemma \ref{disjoint} there is a log resolution $f\colon Y\rightarrow X$
such that the components of $\{\B_Y\}$ do not intersect, and denote $D_i'=k_i(K_Y+\B_{iY})$. Observe that
\begin{equation}\label{eq:4}
R(X;D_1,\dots,D_\ell)\simeq R(Y;D_1',\dots,D_\ell').
\end{equation}
Since $B_i\leq\nu\sum_kF_k$, by comparing discrepancies we see that the
components of $\{\B_{iY}\}$ do not intersect for every $i$, and notice that
$f^*A=f_*^{-1}A\leq\B_{iY}$ for every $i$ since $A$ is very general.
For $s=\sum_{i=1}^\ell t_ie_i\in\mcal S_\Q$ and $t_s=\sum_{i=1}^\ell t_ik_i$, denote $\Delta_s'=\sum_{i=1}^\ell t_ik_i\B_{iY}/t_s$.
Let $H$ be a small effective $f$-exceptional $\Q$-divisor
such that $A'\sim_\Q f^*A-H$ is a general ample $\Q$-divisor, and let $T=f_*^{-1}S$.
Then, setting $\Psi_s=\Delta_s'-f^*A-T+H\geq0$ and $\Omega_s'=\Psi_{s|T}+A_{|T}'$,
the pair $(T,\Omega_s'+A_{|T}')$ is terminal and $K_Y+T+\Psi_s+A'\sim_\Q K_Y+\Delta_s'$. Now replace $X$ by $Y$, $S$ by $T$, $\Delta_s$
by $T+\Psi_s+A'$ and $\Omega_s$ by $\Omega_s'$.\\[2mm]
{\em Step 3.\/}
For every $s\in\mcal S$, denote $F_s=\frac{1}{t_s}\Fix|t_s(K_X+\Delta_s)|_S$ and $F_s^\sharp=\liminf\limits_{m\rightarrow\infty}F_{ms}$.
Define the maps $\Theta\colon\mcal S\rightarrow\Div(S)_\Q$ and $\Theta^\sharp\colon\mcal S\rightarrow \Div(S)_\Q$ by
$$\Theta(s)=\Omega_s-\Omega_s\wedge F_s,\qquad\Theta^\sharp(s)=\Omega_s-\Omega_s\wedge F_s^\sharp.$$
Then, denoting $\Theta_s=\Theta(s)$ and $\Theta_s^\sharp=\Theta^\sharp(s)$, we have
\begin{equation}\label{eq:5}
\res_S R(X;D_1,\dots,D_\ell)\simeq\bigoplus_{s\in\mcal S}H^0(S,t_s(K_S+\Theta_s))
\end{equation}
by Corollary \ref{cor:hmck}. Furthermore, for $s\in\mcal S$ let $\varepsilon>0$ be a rational number such that
$\varepsilon(K_X+\Delta_s)+A$ is ample. Then by Theorem \ref{thm:hmck} we have
$$|k_s(K_S+\Omega_s-\Phi_s)|+k_s\Phi_s\subset|k_s(K_X+\Delta_s)|_S$$
for any $\Phi_s$ and $k_s$ such that $k_s\Delta_s,k_s\Phi_s\in\Div(X)$ and $\Omega_s\wedge(1-\frac{\varepsilon}{k_s})F_s\leq\Phi_s\leq\Omega_s$.
Then similarly as in the proof of \cite[Theorem 7.1]{HM08}, by Lemma \ref{bounded} we have that $\Omega_s\wedge F_s^\sharp$ is rational and
\begin{equation}\label{eq:restriction}
\res_S R(X,k_s^\sharp(K_X+\Delta_s))\simeq R(S,k_s^\sharp(K_S+\Theta_s^\sharp)),
\end{equation}
where $k_s^\sharp\Theta_s^\sharp$ and $k_s^\sharp\Delta_s$ are both Cartier.
Note also, by the same proof, that $G\not\subset\B(K_S+\Theta_s^\sharp)$ for every component $G$ of $\Theta_s^\sharp$. In particular,
$\Theta_{k_s^\sharp ps}=\Theta_{k_s^\sharp s}=\Theta_s^\sharp$ for every $p\in\N$.

Define maps $\lambda\colon\mcal S\rightarrow\Div(S)_\Q$ and $\lambda^\sharp\colon\mcal S\rightarrow \Div(S)_\Q$ by
$$\lambda(s)=t_s(K_S+\Theta_s),\qquad\lambda^\sharp(s)=t_s(K_S+\Theta^\sharp_s).$$
By Theorem \ref{lem:PL} below, there is a finite rational polyhedral subdivision $\mcal S_\R=\bigcup\mcal C_i$ such that the map
$\lambda^\sharp$ is linear on each $\mcal C_i$. In particular, there is a sufficiently divisible positive integer $\kappa$ such that
$\kappa\lambda^\sharp(s)$ is Cartier for every $s\in\mcal S$, and thus $\kappa\lambda^\sharp(s)=\lambda(\kappa s)$ for every $s\in\mcal S$.
Therefore the restriction of $\lambda$ to $\mcal S_i^{(\kappa)}$ is additive, where $\mcal S_i=\mcal S\cap\mcal C_i$.
If $s_1^i,\dots,s_z^i$ are generators of $\mcal S_i^{(\kappa)}$,
then the Cox ring $R(S;\lambda(s_1^i),\dots,\lambda(s_z^i))$ is finitely generated by Theorem \ref{thm:main},
and so is the algebra $R(S,\lambda(\mcal S_i^{(\kappa)}))$ by projection.
Hence the algebra $\bigoplus_{s\in\mcal S}H^0(S,\lambda(s))$ is finitely generated, and this together with \eqref{eq:5} finishes the proof.
\end{proof}

It remains to prove that the map $\lambda^\sharp$ is rationally piecewise linear.
Firstly we have the following result, which can be viewed as a global version of Lemma \ref{bounded}. Recall that $\mcal S=\bigoplus_{i=1}^\ell\N e_i$.

\begin{lem}\label{uniformbound}
There is a positive integer $r$ such that the following stands. If $\Psi\in\Div(S)_\Q$ is such that
$\Supp\Psi\subset\bigcup_{i=1}^\ell\Supp(\Omega_{e_i}-A_{|S})$ and no component of $\Psi$ is in $\B(K_S+\Psi+A_{|S})$, then
no component of $\Psi$ is in $\Fix|k(K_S+\Psi+A_{|S})|$ for every $k$ with $k(\Psi+A_{|S})/r$ Cartier.
\end{lem}
\begin{proof}
Let $\sum_{j=1}^q G_j=\bigcup_{i=1}^\ell\Supp(\Omega_{e_i}-A_{|S})$, and for each $j$ let
$\mcal P_{G_j}=\{\Xi\in\sum_j[0,1]G_j:G_j\not\subset\B(K_S+\Xi+A_{|S})\}$. Each $\mcal P_{G_j}$ is a rational polytope by Property $\mcal L_A^G$.
Let $K_S$ be a divisor such that $\OO_S(K_S)\simeq\omega_S$ and $\Supp A\not\subset\Supp K_X$,
let $\mcal P$ be the convex hull of all rational polytopes $K_S+A_{|S}+\mcal P_{G_j}$, and set $\mcal C=\R_+\mcal P$.
Observe that $K_S+\Psi+A_{|S}\in\mcal C$.
Let $G_{q+1},\dots,G_w$ be the components of $K_S+A_{|S}$ not equal to $G_j$ for $j=1,\dots,q$, and let $\Lambda=\bigoplus_{j=1}^w\N G_j$.
Then by Theorem \ref{thm:main} in dimension $\dim S$ the algebra $R(S,\mcal C\cap\Lambda)$ is
finitely generated and the map $\bMob_{\iota|\mcal C\cap\Lambda^{(r)}}$ is piecewise additive for some $r$ by Proposition \ref{pro:1},
where $\iota\colon\Lambda\rightarrow\Lambda$ is the identity map. In particular,
if $G_j\not\subset\B(K_S+\Psi+A_{|S})$ and $k(\Psi+A_{|S})/r$ is Cartier, then $G_j\not\subset\Fix|k(K_S+\Psi+A_{|S})|$.
\end{proof}

\begin{thm}\label{lem:lipschitz}
For any $s,t\in\mcal S_\R$ we have
$$\lim_{\varepsilon\downarrow0}\Theta_{s+\varepsilon(t-s)}^\sharp=\Theta_s^\sharp.$$
\end{thm}
\begin{proof}
{\em Step 1.}
First we will prove that $\Theta^\sigma_s=\Theta_s^\sharp$, where
$$\Theta^\sigma_s=\Omega_s-\Omega_s\wedge N_\sigma\|K_X+\Delta_s\|_S,$$
cf.\ Remark \ref{rem:11}. I am closely following the proof of \cite[Theorem 7.16]{Hac08}. Let $r$ be a positive integer as in Lemma \ref{uniformbound},
let $\phi<1$ be the smallest positive coefficient of $\Omega_s-\Theta_s^\sigma$ if it exists, and set $\phi=1$ otherwise.
Let $V\subset\Div(X)_\R$ and $W\subset\Div(S)_\R$ be the smallest rational affine spaces containing $\Delta_s$ and $\Theta_s^\sigma$ respectively.
Let $0<\eta\ll1$ be a rational number such that $\eta(K_X+\Delta_s)+\frac12A$ is ample, and if $\Delta'\in V$ with $\|\Delta'-\Delta_s\|<\eta$,
then $\Delta'-\Delta_s+\frac12A$ is ample.
Then by Lemma \ref{lem:surround} there are rational points $(\Delta_i,\Theta_i)\in V\times W$ and integers $k_i\gg0$ such that:
\begin{enumerate}
\item we may write $\Delta_s=\sum r_i\Delta_i$ and $\Theta_s^\sigma=\sum r_i\Theta_i$, where $r_i>0$ and $\sum r_i=1$,
\item $k_i\Delta_i/r$ are integral and $\|\Delta_s-\Delta_i\|<\phi\eta/2k_i$,
\item $k_i\Theta_i/k_s$ are integral, $\|\Theta_s^\sigma-\Theta_i\|<\phi\eta/2k_i$ and observe that $\Theta_i\leq\Omega_i$
since $k_i\gg0$ and $(\Delta_i,\Theta_i)\in V\times W$.
\end{enumerate}
{\em Step 2.}
Set $A_i=A/k_i$ and $\Omega_i=(\Delta_i-S)_{|S}$. In this step I prove that for any component $P\in\Supp\Omega_s$, and for any $l>0$
sufficiently divisible, we have
\begin{equation}\label{equ:1}
\mult_P(\Omega_i\wedge\textstyle\frac1l\Fix|l(K_X+\Delta_i+A_i)|_S)\leq\mult_P(\Omega_i-\Theta_i).
\end{equation}
If $\phi=1$, \eqref{equ:1} follows immediately from Lemma \ref{lem:restrictedord}. Now assume $0<\phi<1$.
Since $\|\Omega_s-\Omega_i\|<\phi\eta/2k_i$ and $\|\Theta_s^\sigma-\Theta_i\|<\phi\eta/2k_i$, it suffices to show that
$$\mult_P(\Omega_i\wedge\textstyle\frac1l\Fix|l(K_X+\Delta_i+A_i)|_S)\leq(1-\frac{\eta}{k_i})\mult_P(\Omega_s-\Theta_s^\sigma).$$
Let $\delta>\eta/k_i$ be a rational number such that $\delta(K_X+\Delta_i)+\frac12A_i$ is ample. Since
$$\textstyle K_X+\Delta_i+A_i=(1-\delta)(K_X+\Delta_i+\frac12A_i)+\big(\delta(K_X+\Delta_i)+\frac{1+\delta}{2}A_i\big),$$
we have
$$\textstyle\ord_P\|K_X+\Delta_i+A_i\|_S\leq(1-\delta)\ord_P\|K_X+\Delta_i+\frac12A_i\|_S,$$
and thus
$$\mult_P\textstyle\frac1l\Fix|l(K_X+\Delta_i+A_i)|_S\leq(1-\frac{\eta}{k_i})\sigma_P\|K_X+\Delta_i\|_S$$
for $l$ sufficiently divisible, cf.\ Lemma \ref{lem:restrictedord}.\\[2mm]
{\em Step 3.}
In this step we prove that there exists an effective divisor $H'$ on $X$ not containing $S$ such that for all sufficiently divisible
positive integers $m$ we have
\begin{multline}\label{eq:6}
|m(K_S+\Theta_i)|+m(\Omega_i-\Theta_i)+(mA_i+H')_{|S}\\
\subset|m(K_X+\Delta_i)+mA_i+H'|_S.
\end{multline}
First observe that since $S\not\subset\B(K_X+\Delta_s)$ and $\Delta_i-\Delta_s+A_i$ is ample, we have $S\not\subset\Bs|m(K_X+\Delta_i+A_i)|$ for
$m$ sufficiently divisible. Assume further that $m$ is divisible by $l$, for $l$ as in Step 2. Let $f\colon Y\rightarrow X$ be a log resolution of
$(X,\Delta_i+A_i)$ and of $|m(K_X+\Delta_i+A_i)|$. Let $\Gamma=\B(X,\Delta_i+A_i)_Y$ and $E=K_Y+\Gamma-f^*(K_X+\Delta_i+A_i)$, and define
$$\Xi=\Gamma-\Gamma\wedge\textstyle\frac1m\Fix|m(K_Y+\Gamma)|.$$
We have that $m(K_Y+\Xi)$ is Cartier, $\Fix|m(K_Y+\Xi)|\wedge\Xi=0$ and $\Mob(m(K_Y+\Xi))$ is free. Since $\Fix|m(K_Y+\Xi)|+\Xi$ has simple
normal crossings support, it follows that $\B(K_Y+\Xi)$ contains no log canonical centres of $(Y,\lceil\Xi\rceil)$. Let
$T=f_*^{-1}S,\Gamma_T=(\Gamma-T)_{|T}$ and $\Xi_T=(\Xi-T)_{|T}$, and consider a section
$$\sigma\in H^0(T,\OO_T(m(K_T+\Xi_T)))=H^0(T,\mcal J_{\|m(K_T+\Xi_T)\|}(m(K_T+\Xi_T))).$$
By \cite[Theorem 5.3]{HM08}, there is an ample divisor $H$ on $Y$ such that if $\tau\in H^0(T,\OO_T(H))$, then $\sigma\cdot\tau$
is in the image of the homomorphism
$$H^0(Y,\OO_Y(m(K_Y+\Xi)+H))\rightarrow H^0(T,\OO_T(m(K_Y+\Xi)+H)).$$
Therefore
\begin{equation}\label{eq:inclusion}
|m(K_T+\Xi_T)|+m(\Gamma_T-\Xi_T)+H_{|T}\subset|m(K_Y+\Gamma)+H|_T.
\end{equation}
We claim that
\begin{equation}\label{eq:inequality}
\Omega_i+A_{i|S}\geq(f_{|T})_*\Xi_T\geq\Theta_i+A_{i|S}
\end{equation}
and so, as $(S,\Omega_i+A_{i|S})$ is canonical, we have
\begin{multline*}
|m(K_S+\Theta_i)|+m((f_{|T})_*\Xi_T-\Theta_i)\\
\subset|m(K_S+(f_{|T})_*\Xi_T)|=(f_{|T})_*|m(K_T+\Xi_T)|.
\end{multline*}
Pushing forward the inclusion \eqref{eq:inclusion}, we obtain \eqref{eq:6} for $H'=f_*H$.

We will now prove the inequality \eqref{eq:inequality} claimed above. We have $\Xi_T\leq\Gamma_T$ and $(f_{|T})_*\Gamma_T=\Omega_i+A_{i|S}$
and so the first inequality follows.

In order to prove the second inequality, let $P$ be any prime divisor on $S$ and let $P'=(f_{|T})^{-1}_*P$. Assume that
$P\subset\Supp\Omega_s$, and thus $P'\subset\Supp\Gamma_T$. Then there is a component $Q$ of the support of $\Gamma$ such that
$$\mult_{P'}\Fix|m(K_Y+\Gamma)|_T=\mult_Q\Fix|m(K_Y+\Gamma)|$$
and $\mult_{P'}\Gamma_T=\mult_Q\Gamma$. Therefore
$$\mult_{P'}\Xi_T=\mult_{P'}\Gamma_T-\min\{\mult_{P'}\Gamma_T,\mult_{P'}\textstyle\frac1m\Fix|m(K_Y+\Gamma)|_T\}.$$
Notice that $\mult_{P'}\Gamma_T=\mult_P(\Omega_i+A_{i|S})$ and since $E_{|T}$ is exceptional, we have that
$$\mult_{P'}\Fix|m(K_Y+\Gamma)|_T=\mult_P\Fix|m(K_X+\Delta_i+A_i)|_S.$$
Therefore $(f_{|T})_*\Xi_T=\Omega_i+A_{i|S}-\Omega_i\wedge\frac1m\Fix|m(K_X+\Delta_i+A_i)|_S$. The inequality now follows from Step 2.\\[2mm]
{\em Step 4.}
In this step we prove
\begin{equation}\label{eq:11}
|k_i(K_S+\Theta_i)|+k_i(\Omega_i-\Theta_i)\subset|k_i(K_X+\Delta_i)|_S.
\end{equation}
For any $\Sigma\in|k_i(K_S+\Theta_i)|$ and any $m>0$ sufficiently divisible, we may choose a divisor $G\in|m(K_X+\Delta_i)+mA_i+H|$
such that $G_{|S}=\frac{m}{k_i}\Sigma+m(\Omega_i-\Theta_i)+(mA_i+H)_{|S}$. If we define $\Lambda=\frac{k_i-1}{m}G+\Delta_i-S-A$, then
$$k_i(K_X+\Delta_i)\sim_\Q K_X+S+\Lambda+A_i-\textstyle\frac{k_i-1}{m}H,$$
where $A_i-\frac{k_i-1}{m}H$ is ample as $m\gg0$. By \cite[Lemma 4.4(3)]{HM08}, we have a surjective homomorphism
$$H^0(X,\mcal J_{S,\Lambda}(k_i(K_X+\Delta_i)))\rightarrow H^0(S,\mcal J_{\Lambda_{|S}}(k_i(K_X+\Delta_i))).$$
Since $(S,\Omega_i)$ is canonical, $(S,\Omega_i+\frac{k_i-1}{m}H_{|S})$ is klt as $m\gg0$, and therefore
$\mcal J_{\Omega_t+\frac{k_i-1}{m}H_{|S}}=\OO_S$. Since
\begin{multline*}
\Lambda_{|S}-(\Sigma+k_i(\Omega_i-\Theta_i))\\
=\textstyle\frac{k_i-1}{m}G_{|S}+\Omega_i-A_{|S}-(\Sigma+k_i(\Omega_i-\Theta_i))\leq\Omega_i+\frac{k_i-1}{m}H_{|S},
\end{multline*}
then by \cite[Lemma 4.3(3)]{HM08} we have $\mcal I_{\Sigma+k_i(\Omega_i-\Theta_i)}\subset\mcal J_{\Lambda_{|S}}$, and so
$$\Sigma+k_i(\Omega_i-\Theta_i)\in|k_i(K_X+\Delta_i)|_S,$$
which proves \eqref{eq:11}.\\[2mm]
{\em Step 5.}
There are ample divisors $A_n$ with $\Supp A_n\subset\Supp(\Delta_s-S)$
such that $\|A_n\|\rightarrow0$ and $\Delta_s+A_n$ are $\Q$-divisors. Observe that $\Theta_s^\sigma=\lim\limits_{n\rightarrow\infty}\Theta_n^\sigma$ with
$$\Theta_n^\sigma=\Omega_n-\Omega_n\wedge N_\sigma\|K_X+\Delta_n\|_S,$$
where $\Delta_n=\Delta_s+A_n$ and $\Omega_n=(\Delta_n-S)_{|S}$. Note that
$$N_\sigma\|K_X+\Delta_n\|_S=\sum\ord_P\|K_X+\Delta_n\|_S\cdot P$$
for all prime divisors $P$ on $S$ for all $n$, cf.\ Remark \ref{rem:11}. But then as in Step 3 of the proof of Theorem \ref{thm:2},
no component of $\Theta_n^\sigma$ is in $\B(K_S+\Theta_n^\sigma)$, and thus, by Property $\mcal L_A^G$ and since $\Theta_n^\sigma\geq\Theta_s^\sigma$
for every $n$, no component of $\Theta_s^\sigma$ is in
$\B(K_S+\Theta_s^\sigma)$. Since $k_i$ is divisible by $r$ and $\Theta_i\in W$, by \eqref{eq:11} we have
$$\Omega_i-\Theta_i\geq\Omega_i\wedge\textstyle\frac{1}{k_i}\Fix|k_i(K_X+\Delta_i)|_S\geq\Omega_i-\Theta_i^\sharp,$$
and so $\Theta_i^\sharp\geq\Theta_i$, where
$$\Theta_i^\sharp=\Omega_i-\Omega_i\wedge\liminf\limits_{m\rightarrow\infty}\textstyle\frac1m\Fix|m(K_X+\Delta_i)|_S.$$
Let $P$ be a prime divisor on $S$. If $\mult_P\Theta_s^\sigma=0$, then $\mult_P\Theta_s^\sharp=0$ since $\Theta_s^\sigma\geq\Theta_s^\sharp$
by Lemma \ref{lem:restrictedord}. Otherwise $\mult_P\Theta_i>0$ for all $i$ and thus $\mult_P\Theta_i^\sharp>0$.
Therefore by concavity we have
$$\mult_P\Theta_s^\sharp\geq\sum r_i\mult_P\Theta_i^\sharp\geq\sum r_i\mult_P\Theta_i=\mult_P\Theta_s^\sigma,$$
proving the claim from Step 1.\\[2mm]
{\em Step 6.}
Now let $C$ be an ample $\Q$-divisor such that $\Delta_t-\Delta_s+C$ is ample. Then by the claim from Step 1 and by Lemma \ref{lem:restrictedord},
\begin{multline*}
\Omega_s-\Theta_s^\sharp=\Omega_s\wedge\lim_{\varepsilon\downarrow0}\big(\sum\ord_P\|K_X+\Delta_s+\varepsilon(\Delta_t-\Delta_s+C)\|_S\cdot P\big)\\
\leq\Omega_s\wedge\lim_{\varepsilon\downarrow0}\big(\sum\ord_P\|K_X+\Delta_s+\varepsilon(\Delta_t-\Delta_s)\|_S\cdot P\big)\leq\Omega_s-\Theta_s^\sharp,
\end{multline*}
where the last inequality follows from convexity. Therefore that inequality is an equality, and this completes the proof.
\end{proof}

Now, let $Z$ be a prime divisor on $S$ and let $\mcal L_Z$ be the closure in $\mcal S_\R$ of the set $\{s\in\mcal S_\R:\mult_Z\Theta_s^\sharp>0\}$.
Then $\mcal L_Z$ is a closed cone. Let $\lambda_Z^\sharp\colon\mcal S_\R\rightarrow\R$ be the function given by
$\lambda_Z^\sharp(s)=\mult_Z\lambda^\sharp(s)$, and similarly for $\Theta_Z^\sharp$.

\begin{thm}\label{lem:PL}
For every prime divisor $Z$ on $S$, the map $\lambda_Z^\sharp$ is
rationally piecewise linear. Therefore, $\lambda^\sharp$ is rationally piecewise linear.
\end{thm}
\begin{proof}
Let $G_1,\dots,G_w$ be prime divisors on $X$ not equal to $S$ and $\Supp A$
such that $\Supp(\Delta_s-S-A)\subset\sum G_i$ for
every $s\in\mcal S$. Let $\nu=\max\{\mult_{G_i}\Delta_s:s\in\mcal S,i=1,\dots,w\}<1$, and let $0<\eta\ll1-\nu$ be a rational number such that
$A-\eta\sum G_i$ is ample. Let $A'\sim_\Q A-\eta\sum G_i$ be a general ample $\Q$-divisor. Define $\Delta_s'=\Delta_s-A+\eta\sum G_i+A'\geq0$,
and observe that $\Delta_s'\sim_\Q\Delta_s$, $\lfloor\Delta_s'\rfloor=S$ and $(S,(\Delta_s'-S)_{|S})$ is terminal.

Define the map $\chi\colon\mcal S\rightarrow\Div(X)$ by $\chi(s)=\kappa t_s(K_X+\Delta_s')$, for $\kappa$ sufficiently divisible.
Then as before, we can construct maps  $\tilde\Theta^\sharp\colon\mcal S_\R\rightarrow\Div(S)_\R$,
$\tilde\lambda^\sharp\colon\mcal S_\R\rightarrow\Div(S)_\R$ and $\tilde\lambda_Z^\sharp\colon\mcal S_\R\rightarrow\R$ associated to $\chi$.
By construction, $\ord_E\|\tilde\lambda_s^\sharp/\kappa t_s\|_S=\ord_E\|\lambda_s^\sharp/t_s\|_S$, and thus
$\mult_Z\tilde\Theta^\sharp_s=\mult_Z\Theta^\sharp_s+\eta$ for every $s\in\mcal L_Z$.
Let $\tilde{\mcal L}_Z$ be the closure in $\mcal S_\R$ of the set
$\{s\in\mcal S_\R:\mult_Z\tilde\Theta_s^\sharp>0\}$, and thus $\mcal L_Z$ is the closure in $\mcal S_\R$ of the set
$\{s\in\mcal S_\R:\mult_Z\tilde\Theta_s^\sharp>\eta\}$. Note that $\mult_Z\tilde\Theta_s^\sharp\geq\eta$ for every $s\in\mcal L_Z$ by
Theorem \ref{lem:lipschitz}. Now for every face $\mcal F$ of $\mcal S_\R$, either
$\mcal F\cap\mcal L_Z\subset\relint(\mcal F\cap\tilde{\mcal L}_Z)$ or
$\partial(\mcal F\cap\mcal L_Z)\cap\partial(\mcal F\cap\tilde{\mcal L}_Z)\subset\partial\mcal F$. Therefore by compactness
there is a rational polyhedral cone $\mcal M_Z$ such that $\mcal L_Z\subset\mcal M_Z\subset\tilde{\mcal L}_Z$,
and so the map $\tilde\lambda_Z^\sharp|_{\mcal M_Z}$ is {\em superlinear.}

By Theorem \ref{thm:linear1} below, for any $2$-plane $H\subset\R^\ell$ the map
$\tilde\lambda_Z^\sharp|_{\mcal M_Z\cap H}$ is piecewise linear, and thus
$\tilde\lambda_Z^\sharp|_{\mcal M_Z}$ is piecewise linear by \cite[Lemma 3.8]{Laz07}.

To prove that $\tilde\lambda_Z^\sharp|_{\mcal M_Z}$ is rationally piecewise linear, let $k=\dim\mcal M_Z$ and let $\mcal M_Z=\bigcup\mcal C_m$ be a
finite polyhedral decomposition such that $\tilde\lambda_Z^\sharp|_{\mcal C_m}$ is linear for every $m$.
Let $\mcal H$ be a hyperplane which contains a common $(k-1)$-dimensional face of cones $\mcal{C}_i$
and $\mcal{C}_j$ and assume $\mcal H$ is not rational.
By Step 1 of the proof of \cite[Lemma 3.5]{Laz07} there is a point $s\in\mcal{C}_i\cap\mcal{C}_j$ such that the minimal affine rational space
containing $s$ has dimension $k-1$.
Then as in Step 1 of the proof of Theorem \ref{thm:linear1} there is an $k$-dimensional cone $\widetilde{\mcal{C}}$ such that
$s\in\Int\widetilde{\mcal{C}}$ and the map $\tilde\lambda_Z^\sharp|_{\widetilde{\mcal{C}}}$ is linear.
But then the cones $\widetilde{\mcal{C}}\cap\mcal{C}_i$ and $\widetilde{\mcal{C}}\cap\mcal{C}_i$ are $k$-dimensional and linear
extensions of $\tilde\lambda_Z^\sharp|_{\mcal{C}_i}$ and $\tilde\lambda_Z^\sharp|_{\mcal{C}_j}$
coincide since they are equal to the linear extension of $\tilde\lambda_Z^\sharp|_{\widetilde{\mcal{C}}}$, a contradiction.
Therefore all $(k-1)$-dimensional faces of the cones $\mcal{C}_i$ belong to rational hyperplanes and thus $\mcal{C}_i$ are rational
cones.

Therefore the map $\tilde\lambda_Z^\sharp|_{\mcal M_Z}$ is rationally piecewise linear, and since $\mcal L_Z$ is the closure of the set
$\{s\in\mcal S_\R:\mult_Z\tilde\Theta_s^\sharp>\eta\}$, we have that $\mcal L_Z$ is a rational polyhedral cone,
the map $\tilde\lambda_Z^\sharp|_{\mcal L_Z}$ is rationally piecewise linear, and therefore so is $\lambda_Z^\sharp$. Now it is trivial that
$\lambda^\sharp$ is a rationally piecewise linear map.
\end{proof}
Thus it remains to prove that $\lambda^\sharp_Z|_{\mcal M_Z\cap H}$ is piecewise linear for every $2$-plane $H\subset\R^\ell$.
As in Step 1 of the proof of Theorem \ref{thm:2}, by replacing $\mcal S_\R$ by $\mcal M_Z$ and $\lambda^\sharp_Z$ by $\tilde\lambda^\sharp_Z$,
it is enough to assume, and I will until the end of the section, that $\lambda_Z^\sharp$
is a superlinear function on $\mcal S_\R$ for a fixed prime divisor $Z$ on $S$.

Let $C_s$ be a local Lipschitz constant of $\Theta^\sharp$ around $s\in\mcal S_\R$ in the smallest rational affine space containing $s$.
For every $s\in\mcal S$, let $\phi_s$ be the smallest coefficient of $\Omega_s-\Theta_s^\sharp$.

\begin{thm}\label{thm:inclusion}
Fix $s\in\mcal S_\R$ and let $U\subset\R^\ell$ be the smallest rational affine subspace containing s.
If $\phi_s>0$, let $0<\delta\ll1$ be a rational number such that $\phi_u>0$ for $u\in U$ with
$\|u-s\|\leq\delta$, set $\phi=\min\{\phi_u:u\in U,\|u-s\|\leq\delta\}$ and let $0<\varepsilon\ll\delta$ be a rational number such that
$(C_s/\phi+1)\varepsilon(K_X+\Delta_s)+A$ is ample.
If $\phi_s=0$ and $\Supp\Delta_s=\sum F_i$, let $0<\varepsilon\ll1$ be a rational number such that $\sum f_iF_i+A$ is ample
for any $f_i\in(-\varepsilon,\varepsilon)$, and set $\phi=1$. Let $t\in U\cap\mcal S_\Q$ and $k_t\gg0$ be an integer such that
$\|t-s\|<\varepsilon/k_t$, $k_t\Delta_t/r$ is Cartier for $r$ as in Lemma \ref{uniformbound} and $S\not\subset\B(K_X+\Delta_t)$.
Then for any divisor $\Theta$ on $S$ such that $\Theta\leq\Omega_t$, $\|\Theta-\Theta_s^\sharp\|<\phi\varepsilon/k_t$ and $k_t\Theta/r$ is Cartier we have
$$|k_t(K_S+\Theta)|+k_t(\Omega_t-\Theta)\subset|k_t(K_X+\Delta_t)|_S.$$
\end{thm}
\begin{proof}
Set $A_t=A/k_t$. I first prove that for any component $P\in\Supp\Omega_s$, and for any $l>0$ sufficiently divisible, we have
\begin{equation}\label{equ:2}
\mult_P(\Omega_t\wedge\textstyle\frac1l\Fix|l(K_X+\Delta_t+A_t)|_S)\leq\mult_P(\Omega_t-\Theta).
\end{equation}
Assume first that $\phi_s=0$. Then in particular $\ord_P\|K_X+\Delta_s\|_S=0$ and $\Delta_t-\Delta_s+A_t$ is ample since
$\|\Delta_t-\Delta_s\|<\varepsilon/k_t$, so
\begin{align*}
\ord_P\|K_X+\Delta_t+A_t\|_S&=\ord_P\|K_X+\Delta_s+(\Delta_t-\Delta_s+A_t)\|_S\\
&\leq\ord_P\|K_X+\Delta_s\|_S=0.
\end{align*}
Since for $l$ sufficiently divisible we have
\begin{equation}\label{equ:3}
\mult_P\textstyle\frac1l\Fix|l(K_X+\Delta_t+A_t)|_S=\ord_P\|K_X+\Delta_t+A_t\|_S
\end{equation}
as in Step 3 of the proof of Theorem \ref{thm:2}, we obtain \eqref{equ:2}.

Now assume that $\phi_s\neq0$ and set $C=C_s/\phi$. By Lipschitz continuity we have $\|\Theta_t^\sharp-\Theta_s^\sharp\|<C\phi\varepsilon/k_t$, so
$\|\Theta_t^\sharp-\Theta\|<(C+1)\phi\varepsilon/k_t$. Therefore it suffices to show that
$$\mult_P(\Omega_t\wedge\textstyle\frac1l\Fix|l(K_X+\Delta_t+A_t)|_S)\leq(1-\frac{C+1}{k_t}\varepsilon)\mult_P(\Omega_t-\Theta_t^\sharp).$$
Since $k_t\gg0$, we can choose a rational number $\eta>(C+1)\varepsilon/k_t$ such that $\eta(K_X+\Delta_t)+A_t$ is ample. From
$$K_X+\Delta_t+A_t=(1-\eta)(K_X+\Delta_t)+(\eta(K_X+\Delta_t)+A_t),$$
we have
$$\ord_P\|K_X+\Delta_t+A_t\|_S\leq(1-\eta)\ord_P\|K_X+\Delta_t\|_S,$$
and thus by \eqref{equ:3},
$$\mult_P\textstyle\frac1l\Fix|l(K_X+\Delta_t+A_t)|_S\leq(1-\frac{C+1}{k_t}\varepsilon)\ord_P\|K_X+\Delta_t\|_S$$
for $l$ sufficiently divisible.

Now the theorem follows as in Steps 3 and 4 of the proof of Theorem \ref{lem:lipschitz}.
\end{proof}

Finally, we have

\begin{thm}\label{thm:linear1}
Fix $s\in\mcal{S}_\R$ and let $R$ be a ray in $\mcal{S}_\R$ not containing $s$.
Then there exists a ray $R'\subset\R_+s+R$ not containing $s$ such that the map $\lambda_Z^\sharp|_{\R_+s+R'}$ is linear.
In particular, for every $2$-plane $H\subset\R^\ell$, the map $\lambda_Z^\sharp|_{\mcal S_\R\cap H}$ is piecewise linear.
\end{thm}
\begin{proof}
{\em Step 1.}
Let $U\subset\R^\ell$ be the smallest rational affine space containing $s$. In this step I prove that the map
$\Theta^\sharp$ is linear in a neighbourhood of $s$ contained in $U$.

Let $\varepsilon$ and $\phi$ be as in Theorem \ref{thm:inclusion}.
Let $W\subset\R^\ell$ and $V\subset\Div(S)_\R$ be the smallest rational affine spaces containing $s$ and $\Theta_s^\sharp$ respectively,
and let $r$ be as in Lemma \ref{uniformbound}. By Lemma \ref{lem:surround}, there exist rational points
$(t_i,\Theta_{t_i}')\in W\times V$ and integers $k_{t_i}\gg0$ such that:
\begin{enumerate}
\item we may write $s=\sum r_{t_i}t_i$, $\Delta_s=\sum r_{t_i}\Delta_{t_i}$ and $\Theta_s^\sharp=\sum r_{t_i}\Theta_{t_i}'$, where
$r_{t_i}>0$ and $\sum r_{t_i}=1$,
\item $k_{t_i}\Delta_{t_i}/r$ are integral and $\|s-t_i\|<\varepsilon/k_{t_i}$,
\item $k_{t_i}\Theta_{t_i}'/r$ are integral, $\|\Theta_s^\sharp-\Theta_{t_i}'\|<\phi\varepsilon/k_{t_i}$ and note that $\Theta_{t_i}'\leq\Omega_{t_i}$
since $k_{t_i}\gg0$ and $(t_i,\Theta_{t_i}')\in W\times V$.
\end{enumerate}
Observe that $S\not\subset\B(K_X+\Delta_{t_i})$ since $t_i\in W$ for every $i$ and $\varepsilon\ll1$ by Property $\mcal L_A^G$. By local Lipschitz
continuity and by Theorem \ref{thm:inclusion} we
have that
$$|k_{t_i}(K_S+\Theta_{t_i}')|+k_{t_i}(\Omega_{t_i}-\Theta_{t_i}')\subset|k_{t_i}(K_X+\Delta_{t_i})|_S.$$
Since $\Theta_{t_i}'\in V$ and $k_{t_i}\Theta_{t_i}'/r$ is Cartier, no component of $\Theta_{t_i}'$ is in $\Fix|k_{t_i}(K_S+\Theta_{t_i}')|$
for every $i$ by Lemma \ref{uniformbound}. In particular,
$$\Omega_{t_i}-\Theta_{t_i}'\geq\Omega_{t_i}\wedge\textstyle\frac{1}{k_{t_i}}\Fix|k_{t_i}(K_X+\Delta_{t_i})|_S\geq\Omega_{t_i}-\Theta_{t_i}^\sharp,$$
and so
$$\Theta_{t_i}^\sharp\geq\Theta_{t_i}'.$$
But by assumption (1) and since the map $\Theta_Z^\sharp$ is concave, we have
$$\Theta_Z^\sharp(s)\geq\sum r_{t_i}\Theta_Z^\sharp(t_i)\geq\sum r_{t_i}\mult_Z\Theta_{t_i}'=\Theta_Z^\sharp(s),$$
which proves the statement by \cite[Lemma 2.6]{Laz07}.\\[2mm]
{\em Step 2.}
Now assume $s\in\mcal S_\Q$, $\phi_s=0$ and fix $u\in R$ such that $s$ and $u$ belong to a rational affine subspace $\mcal P$ of $\R^\ell$.
Let $\Delta\colon\bigoplus_{i=1}^\ell\R e_i\rightarrow\Div(X)_\R$ be a linear map given by $\Delta(p_i)=\Delta_{p_i}$ for linearly independent
points $p_1,\dots,p_\ell\in\mcal P\cap\mcal S_\Q$, and then extended linearly. Observe that $\Delta(p)=\Delta_p$ for every $p\in\mcal P\cap\mcal S_\R$.

Let $W$ be the smallest rational affine subspace containing $s$ and $u$.
If there is a sequence $s_n\in(s,u]$ such that $\lim_{n\rightarrow\infty}s_n=s$ and $\phi_{s_n}=0$,
then $\lambda^\sharp$ is linear on the cone $\R_+s+\R_+s_1$ by \cite[Lemma 2.6]{Laz07}.

Therefore we can assume that there are rational numbers $0<\varepsilon,\eta\ll1$ such that for all $v\in[s,u]$ with $0<\|v-s\|<2\varepsilon$ we have
$\phi_v>0$, that for every prime divisor $P$ on $S$, we have either $\mult_P\Omega_v>\mult_P\Theta_v^\sharp$ or
$\mult_P\Omega_v=\mult_P\Theta_v^\sharp$ and either $\mult_P\Theta_v^\sharp=0$ or $\mult_P\Theta_v^\sharp>0$
for all such $v$, and that $\Delta_v-\Delta_s+\Xi+A$ is ample for all such $v$ and for any divisor $\Xi$ such that
$\Supp\Xi\subset\Supp\Delta_s\cup\Supp\Delta_u$ and $\|\Xi\|<\eta$.

Pick $t\in(s,u]$ such that $\|s-t\|<\varepsilon/k_s$, $k_ss$ is integral and the smallest rational affine subspace containing $t$ is precisely $W$.
Let $0<\delta\ll1$ be a rational number such that $\phi_v>0$ for
$v\in W$ with $\|v-t\|\leq\delta$, set $\phi=\min\{\phi_v:v\in W,\|v-t\|\leq\delta\}$ and let $0<\xi\ll\min\{\delta,\varepsilon\}$ be a
rational number such that $(C_t/\phi+1)\xi(K_X+\Delta_t)+A$ is ample.
Denote by $V\subset\Div(S)_\R$ the smallest rational affine space containing $\Theta_s^\sharp=\Omega_s$ and $\Theta_t^\sharp$, and let $r$ be as in Lemma
\ref{uniformbound}. Then by Lemma \ref{lem:approximation} there exist rational points
$(t_i,\Theta_{t_i}')\in W\times V$ and integers $k_{t_i}\gg0$ such that:
\begin{enumerate}
\item we may write $t=\sum r_{t_i}t_i$, $\Delta_t=\sum r_{t_i}\Delta_{t_i}$ and $\Theta_t^\sharp=\sum r_{t_i}\Theta_{t_i}'$, where
$r_{t_i}>0$ and $\sum r_{t_i}=1$,
\item $t_1=s$, $\Theta_{t_1}'=\Theta_{t_1}^\sharp=\Omega_{t_1}$, $k_{t_1}=k_s$,
\item $k_{t_i}\Delta_{t_i}/r$ are integral and $\|t-t_i\|<\xi/k_{t_i}$ for $i=2,\dots,n-1$,
\item $\Theta_{t_i}'\leq\Omega_{t_i}$, $k_{t_i}\Theta_{t_i}'/r$ are integral, $\|\Theta_t^\sharp-\Theta_{t_i}'\|<\phi\xi/k_{t_i}$
and $(t_i,\Theta_{t_i}')$ belong to the smallest rational affine space containing $(t,\Theta_t^\sharp)$ for $i=2,\dots,n-1$,
\item $\Delta_t=\frac{k_{t_1}}{k_{t_1}+k_{t_n}}\Delta_{t_1}+\frac{k_{t_n}}{k_{t_1}+k_{t_n}}\Delta_{t_n}+\Psi$,
where $k_{t_n}\Delta_{t_n}/r$ is integral, $\|t-t_n\|<\varepsilon/k_{t_n}$ and $\|\Psi\|<\eta/(k_{t_1}+k_{t_n})$,
\item $\Theta^\sharp_t=\frac{k_{t_1}}{k_{t_1}+k_{t_n}}\Theta_{t_1}'+\frac{k_{t_n}}{k_{t_1}+k_{t_n}}\Theta_{t_n}'+\Phi$,
where $\Theta_{t_n}'\leq\Omega_{t_n}$, $k_{t_n}\Theta_{t_n}'/r$ is integral, $\|\Theta_t^\sharp-\Theta_{t_n}'\|<\varepsilon/k_{t_n}$
and $\|\Phi\|<\eta/(k_{t_1}+k_{t_n})$.
\end{enumerate}
Observe also that $\Supp\Psi\subset\Supp\Delta_t$ and $\Supp\Phi\subset\Supp\Theta_t^\sharp$ by Remarks \ref{rem:2} and \ref{rem:3} applied
to the linear map $\Delta$ defined at the beginning of Step 2. Then by Theorem \ref{thm:inclusion},
$$|k_{t_i}(K_S+\Theta_{t_i}')|+k_{t_i}(\Omega_{t_i}-\Theta_{t_i}')\subset|k_{t_i}(K_X+\Delta_{t_i})|_S$$
for $i=2,\dots,n-1$. Let $P$ be a component in $\Supp\Omega_t$ and denote $A_{t_n}=A/k_{t_n}$. I claim that
\begin{equation}\label{equ:4}
\mult_P(\Omega_{t_n}\wedge\textstyle\frac1l\Fix|l(K_X+\Delta_{t_n}+A_{t_n})|_S)\leq\mult_P(\Omega_{t_n}-\Theta_{t_n}')
\end{equation}
for $l\gg0$ sufficiently divisible. Assume first that $\mult_P\Theta_t^\sharp=0$. Then $\mult_P\Theta_s^\sharp=0$
by the choice of $\varepsilon$, and thus $\mult_P\Theta_{t_n}'=0$ since $\Theta_{t_n}'\in V$. Therefore
\begin{multline*}
\mult_P(\Omega_{t_n}\wedge\textstyle\frac1l\Fix|l(K_X+\Delta_{t_n}+A_{t_n})|_S)\\
\leq\mult_P\Omega_{t_n}=\mult_P(\Omega_{t_n}-\Theta_{t_n}').
\end{multline*}
Now assume that $\mult_P\Theta_t^\sharp>0$. Then for $l$ sufficiently divisible we have
$$\mult_P\textstyle\frac1l\Fix|l(K_X+\Delta_{t_n}+A_{t_n})|_S=\ord_P\|K_X+\Delta_{t_n}+A_{t_n}\|_S$$
as in Step 3 of the proof of Theorem \ref{thm:2}, and since $\Delta_t-\Delta_{t_1}-\frac{k_{t_1}+k_{t_n}}{k_{t_1}}\Psi+A$ is ample by the choice
of $\eta$,
\begin{align*}
\mult_P(\Omega_{t_n}&\wedge\textstyle\frac1l\Fix|l(K_X+\Delta_{t_n}+A_{t_n})|_S)\leq\ord_P\|K_X+\Delta_{t_n}+A_{t_n}\|_S\\
&=\ord_P\big\|K_X+\Delta_t+\textstyle\frac{k_{t_1}}{k_{t_n}}\big(\Delta_t-\Delta_{t_1}-\frac{k_{t_1}+k_{t_n}}{k_{t_1}}\Psi+A\big)\big\|_S\\
&\leq\ord_P\|K_X+\Delta_t\|_S=\mult_P(\Omega_t-\Theta_t^\sharp).
\end{align*}
Combining assumptions (5) and (6) above we have
$$\textstyle\Omega_t-\Theta_t^\sharp\leq\Omega_t-\Theta_t^\sharp+\frac{k_{t_1}}{k_{t_n}}\big(\Omega_t-\Theta_t^\sharp-
\frac{k_{t_1}+k_{t_n}}{k_{t_1}}(\Psi_{|S}-\Phi)\big)=\Omega_{t_n}-\Theta_{t_n}',$$
and \eqref{equ:4} is proved.
Furthermore, we can choose $\varepsilon\ll1$ and $k_{t_n}\gg0$ such that $S\not\subset\B(K_X+\Delta_{t_n})$. Otherwise, if we denote
$\mcal Q=\{p\in\mcal S_\R:S\not\subset\B(K_X+\Delta_p)\}$, $\mcal Q$ is a rational polyhedral cone by Property $\mcal L_A^G$,
and $t\in\partial\mcal Q$ for every $t\in[s,u]$ with $0<\|t-s\|\ll1$, and thus $s\in\partial\mcal Q$. But then
for $0<\|t-s\|\ll1$, $s$ and $t$ belong to the same face of $\mcal Q$, and so does $t_n$, a contradiction. Therefore
as in the proof of Theorem \ref{lem:lipschitz} we have
$$|k_{t_n}(K_S+\Theta_{t_n}')|+k_{t_n}(\Omega_{t_n}-\Theta_{t_n}')\subset|k_{t_n}(K_X+\Delta_{t_n})|_S.$$
Denote $\sum G_j=\Supp(\Omega_s-A_{|S})\cup\Supp(\Omega_u-A_{|S})$, and let
$\mcal Q'=\{\Xi\in\sum_j[0,1]G_j:Z\not\subset\B(K_S+\Xi+A_{|S})\}$. Then by Property $\mcal L_A^G$, $\mcal Q'$ is a rational
polytope and $\Theta_p^\sharp\in\mcal Q'$ for every $p\in\mcal S_\R$.
Therefore as above and by Theorem \ref{lem:lipschitz}, if $\varepsilon\ll1$ then $Z\not\subset\B(K_S+\Theta_{t_n}')$, and as in Step 1
we have that $\lambda_Z^\sharp$ is linear on the cone $\sum_{i=1}^n\R_+t_i$, and in particular on the cone $\R_+s+\R_+t$.\\[2mm]
{\em Step 3.}
Assume now that $s\in\mcal S_\Q$, $\phi_s>0$ and fix $u\in R$. Let again $W$ be the smallest rational affine space containing $s$ and $u$.
Let $0<\xi\ll1$ be a rational number such that $\phi_v>0$ for $v\in[s,u]$ with $\|v-s\|\leq2\xi$, that for every prime divisor $P$ on $S$ we have
either $\mult_P\Omega_v>\mult_P\Theta_v^\sharp$ or $\mult_P\Omega_v=\mult_P\Theta_v^\sharp$ for all such $v$, and let
$\phi=\min\{\phi_v:v\in[s,u],\|v-s\|\leq2\xi\}$.

Let $k_s$ be a positive integer such that $k_s\Delta_s/r$ and $k_s\Theta_s^\sharp/r$ are integral, where $r$ is as in Lemma \ref{uniformbound}.
Let us first show that there is a real number $0<\varepsilon\ll\xi$ such that $(C_t/\phi+1)\varepsilon(K_X+\Delta_v)+A$ is ample for
all $v\in\mcal S_\R$ such that $\|v-s\|<2\xi$, where $\|t-s\|=\varepsilon/k_s$. If $\Theta^\sharp$ is locally Lipschitz around $s$ this is straightforward.
Otherwise, assume $\Theta^\sharp$ is not locally Lipschitz around $s$ and assume we cannot find such $\varepsilon$. But then there is a sequence
$s_n\in(s,u]$ such that $\lim\limits_{n\rightarrow\infty}s_n=s$ and $C_{s_n}\|s_n-s\|\geq M$, where $M$ is a constant and $C_{s_n}\rightarrow\infty$.
Since a local Lipschitz constant is the maximum of local slopes of the concave function $\Theta^\sharp|_{[s,u]}$, we have that
$$\frac{\Theta_{s_n}^\sharp-\Theta_s^\sharp}{\|s_n-s\|}>C_{s_n}.$$
Therefore
$$\Theta_{s_n}^\sharp-\Theta_s^\sharp>C_{s_n}\|s_n-s\|\geq M$$
for all $n\in\N$, which contradicts Theorem \ref{lem:lipschitz}.

Increase $\varepsilon$ a bit, and pick $t\in(s,u]$ such that $\|s-t\|<\varepsilon/k_s$, the smallest rational subspace containing $t$ is precisely $W$
and $(C_t/\phi+1)\varepsilon(K_X+\Delta_v)+A$ is ample for all $v\in\mcal S_\R$ such that $\|v-s\|<2\varepsilon$. In particular, $\Theta^\sharp$ is
locally Lipschitz in a neighbourhood of $t$ contained in $W$. Furthermore, by changing $\phi$ slightly I can assume that
$\phi\leq\min\{\phi_v:v\in W,\|v-t\|\ll1\}$. Denote by $V$ the smallest rational affine space
containing $\Theta_s^\sharp$ and $\Theta_t^\sharp$, and let $r$ be as in Lemma \ref{uniformbound}. Then by Lemma \ref{lem:approximation}
there exist rational points $(t_i,\Theta_{t_i}')\in W\times V$ and integers $k_{t_i}\gg0$ such that:
\begin{enumerate}
\item we may write $t=\sum r_{t_i}t_i$, $\Delta_t=\sum r_{t_i}\Delta_{t_i}$ and $\Theta_t^\sharp=\sum r_{t_i}\Theta_{t_i}'$, where
$r_{t_i}>0$ and $\sum r_{t_i}=1$,
\item $t_1=s$, $\Theta_{t_1}'=\Theta_{t_1}^\sharp$, $k_{t_1}=k_s$,
\item $k_{t_i}\Delta_{t_i}/r$ are integral and $\|t-t_i\|<\varepsilon/k_{t_i}$ for all $i$,
\item $\Theta_{t_i}'\leq\Omega_{t_i}$, $k_{t_i}\Theta_{t_i}'/r$ are integral and $\|\Theta_t^\sharp-\Theta_{t_i}'\|<\phi\varepsilon/k_{t_i}$.
\end{enumerate}
Observe that similarly as in Step 2 we have $S\not\subset\B(K_X+\Delta_{t_i})$ for all $i$, and therefore by Theorem \ref{thm:inclusion},
$$|k_{t_i}(K_S+\Theta_{t_i}')|+k_{t_i}(\Omega_{t_i}-\Theta_{t_i}')\subset|k_{t_i}(K_X+\Delta_{t_i})|_S$$
for all $i$. Then we finish as in Step 2.\\[2mm]
{\em Step 4.}
Assume in this step that $s\in\mcal S_\R$ is a non-rational point and fix $u\in R$.
By Step 1 there is a rational cone $\mcal{C}=\sum_{i=1}^k\R_+g_i$ with $g_i\in\mcal S_\Q$ and $k>1$ such that $\lambda_Z^\sharp$ is
linear on $\mcal{C}$ and $s=\sum\alpha_ig_i$ with all $\alpha_i>0$. Consider the rational point $g=\sum_{i=1}^kg_i$.
Then by Step 2 there is a point $s'=\alpha g+\beta u$ with $\alpha,\beta>0$ such that
the map $\lambda_Z^\sharp$ is linear on the cone $R_+g+\R_+s'$. Now we have
$$\lambda_Z^\sharp\Big(\sum g_i+s'\Big)=\lambda_Z^\sharp(g+s')=\lambda_Z^\sharp(g)+\lambda_Z^\sharp(s')=\sum \lambda_Z^\sharp(g_i)+\lambda_Z^\sharp(s'),$$
so the map $\lambda_Z^\sharp|_{\mcal{C}+\R_+s'}$ is linear by \cite[Lemma 2.6]{Laz07}. Taking $\mu=\max\limits_i\{\frac{\alpha}{\alpha_i\beta}\}$
and taking a point $\hat u=\mu s+u$ in the relative interior of $\R_+s+R$, it is easy to check that
$$\hat u=\sum(\mu\alpha_i-\textstyle\frac\alpha\beta)t_i+\frac1\beta s'\in\mcal{C}+\R_+s',$$
so the map $\lambda_Z^\sharp|_{\R_+s+\R_+\hat u}$ is linear.\\[2mm]
{\em Step 5.}
Finally, let $H$ be any $2$-plane in $\R^\ell$. Then by the previous steps, for every ray $R\subset\mcal S_\R\cap H$ there is a polyhedral cone
$\mcal{C}_R$ with $R\subset\mcal{C}_R\subset\mcal S_\R\cap H$ such that there is a polyhedral decomposition
$\mcal{C}_R=\mcal{C}_{R,1}\cup\mcal{C}_{R,2}$ with $\lambda_Z^\sharp|_{\mcal{C}_{R,1}}$ and $\lambda_Z^\sharp|_{\mcal{C}_{R,2}}$
being linear maps, and if $R\subset\relint(\mcal S_\R\cap H)$, then $R\subset\relint\mcal{C}_R$.

Let $S^{\ell-1}$ be the unit sphere. Restricting to the compact set
$S^{\ell-1}\cap\mcal S_\R\cap H$ we see that $\lambda_Z^\sharp|_{\mcal S_\R\cap H}$ is piecewise linear.
\end{proof}

\section{Proof of the main result}\label{proofmain}

\begin{proof}[Proof of Theorem \ref{thm:main}]\mbox{}\\[2mm]
{\em Step 1.\/}
I first show that it is enough to prove the theorem in the case when $A$ is a general ample $\Q$-divisor and $(X,\Delta_i+A)$ is a log
smooth klt pair for every $i$.

Let $p$ and $k$ be sufficiently divisible positive integers such that all divisors $k(\Delta_i+pA)$ are very ample and $(p+1)kA$ is very ample.
Let $(p+1)kA_i$ be a general section of $|k(\Delta_i+pA)|$ and let $(p+1)kA'$ be a general section of $|(p+1)kA|$.
Set $\Delta_i'=\frac{p}{p+1}\Delta_i+A_i$. Then the pairs $(X,\Delta_i'+A')$ are klt and
$$(p+1)k(K_X+\Delta_i+A)\sim(p+1)k(K_X+\Delta_i'+A')=:D_i'$$
for every $i$. Then a truncation of $R(X;D_1,\dots,D_\ell)$ is isomorphic to $R(X;D_1',\dots,D_\ell')$, so it is enough to prove the latter
algebra is finitely generated.\\[2mm]
{\em Step 2.\/}
Therefore I can assume that $\Delta_i=\sum_{j=1}^N\delta_{ij}F_j$ with $\delta_{ij}\in[0,1)$. Write
$K_X+\Delta_i+A\sim_\Q\sum_{j=1}^Nf_{ij}'F_j\geq0$, where $F_j\neq A$ since $A$ is general.
By blowing up, and by possibly replacing the pair $(X,\Delta_i)$ by $(Y,\Delta_i')$ for some model $Y\rightarrow X$ as in Step 2 of the proof of
Theorem \ref{thm:2}, I can assume that the divisor $\sum_{j=1}^N F_j$ has simple normal crossings. Thus for every $i$,
$$K_X\sim_\Q-A+\sum\nolimits_{j=1}^N f_{ij}F_j,$$
where $f_{ij}=f_{ij}'-\delta_{ij}>-1$.

Let $\Lambda=\bigoplus_{j=1}^N\N F_j\subset\Div(X)$ be a simplicial monoid and denote $\mcal T=\{(t_1,\dots,t_\ell):t_i\geq0,\sum t_i=1\}\subset\R^\ell$.
For each $\tau=(t_1,\dots,t_\ell)\in\mcal T$, denote $\delta_{\tau j}=\sum_i t_i\delta_{ij}$ and $f_{\tau j}=\sum_i t_if_{ij}$, and
observe that $K_X\sim_\R-A+\sum_j f_{\tau j}F_j$. Denote $\mcal B_\tau=\sum_{j=1}^N[\delta_{\tau j}+f_{\tau j},1+f_{\tau j}]F_j\subset\Lambda_\R$
and let $\mcal B=\bigcup_{\tau\in\mcal T}\mcal B_\tau$. It is easy to see that $\mcal B$ is a rational polytope: every point in $\mcal B$
is a barycentric combination of the vertices of $\mcal B_{\tau_1},\dots,\mcal B_{\tau_\ell}$, where $\tau_i$ are the standard basis vectors of
$\R^\ell$. Thus $\mcal C=\R_+\mcal B$ is a rational polyhedral cone.

For each $j=1,\dots,N$ fix a section $\sigma_j\in H^0(X,F_j)$ such that $\ddiv\sigma_j=F_j$.
Consider the $\Lambda$-graded algebra $\mathfrak{R}=\bigoplus_{s\in\Lambda}\mathfrak R_s\subset R(X;F_1,\dots,F_N)$ generated
by the elements of $R(X,\mcal C\cap\Lambda)$ and all $\sigma_j$; observe that $\mathfrak R_s=H^0(X,s)$ for every $s\in\mcal C\cap\Lambda$.
I claim that it is enough to show that $\mathfrak{R}$ is finitely generated.

To see this, assume $\mathfrak{R}$ is finitely generated and denote
$$\omega_i=rk_i\sum\nolimits_j(\delta_{ij}+f_{ij})F_j\in\Lambda$$
for $r$ sufficiently divisible and $i=1,\dots,\ell$. Set $\mcal G=\sum_i\R_+\omega_i\cap\Lambda$ and observe that $\omega_i\sim rD_i$.
Then by Lemma \ref{lem:1}(2) the algebra $R(X,\mcal C\cap\Lambda)$
is finitely generated, and therefore by Proposition \ref{pro:1} there is a finite rational polyhedral subdivision
$\mcal G_\R=\bigcup_k\mcal G_k$ such that the map $\bMob_{\iota|\mcal G_k\cap\Lambda}$ is additive up to truncation for every $k$,
where $\iota\colon\Lambda\rightarrow\Lambda$ is the identity map.

Let $\omega_1',\dots,\omega_q'$ be generators of $\mcal G$ such that $\omega_i'=\omega_i$ for $i=1,\dots,\ell$, and let
$\pi\colon\bigoplus_{i=1}^q\N\omega_i'\rightarrow\mcal G$ be the natural projection.
Then the map $\bMob_{\pi|\pi^{-1}(\mcal G_k\cap\Lambda)}$ is additive up to truncation for every $k$,
and thus $R(X,\pi(\bigoplus_{i=1}^q\N\omega_i'))$ is finitely generated by Lemma \ref{lem:1}(3). Therefore
$R(X,\pi(\bigoplus_{i=1}^\ell\N\omega_i))\simeq R(X;rD_1,\dots,rD_\ell)$ is finitely generated by Lemma \ref{lem:1}(2),
thus $R(X;D_1,\dots,D_\ell)$ is finitely generated by Lemma \ref{lem:1}(1).\\[2mm]
{\em Step 3.\/}
Therefore it suffices to prove that $\mathfrak R$ is finitely generated.
Take a point $\sum_j(f_{\tau j}+b_{\tau j})F_j\in\mcal B\backslash\{0\}$; in particular $b_{\tau j}\in[\delta_{\tau j},1]$.
Setting
$$r_\tau=\max_{j=1}^N\Big\{\frac{f_{\tau j}+b_{\tau j}}{f_{\tau j}+1}\Big\}\quad\text{and}\quad
b_{\tau j}'=-f_{\tau j}+\frac{f_{\tau j}+b_{\tau j}}{r_\tau},$$
we have
\begin{equation}\label{eq:3}
\sum\nolimits_j(f_{\tau j}+b_{\tau j})F_j=r_\tau\sum\nolimits_j(f_{\tau j}+b_{\tau j}')F_j.
\end{equation}
Observe that $r_\tau\in(0,1]$, $b_{\tau j}'\in[b_{\tau j},1]$ and there exists $j_0$ such that $b_{\tau j_0}'=1$.
For every $j=1,\dots,N$, let
$$\mcal F_{\tau j}=(1+f_{\tau j})F_j+\sum\nolimits_{k\neq j}[\delta_{\tau k}+f_{\tau k},1+f_{\tau k}]F_k,$$
and set $\mcal F_j=\bigcup_{\tau\in\mcal T}\mcal F_{\tau j}$, which is a rational polytope. Then $\mcal C_j=\R_+\mcal F_j$ is a rational polyhedral cone,
and \eqref{eq:3} shows that $\mcal C=\bigcup_j\mcal C_j$. Furthermore, since $\sum_j(f_{\tau j}+b_{\tau j})F_j\sim_\R K_X+\sum_jb_{\tau j}F_j+A$ for
$\tau\in\mcal T$, for every $j$ and for every $s\in\mcal C_j\cap\Lambda$ there is $r_s\in\Q_+$ such that
$s\sim_\Q r_s(K_X+F_j+\Delta_s+A)$ where $\Supp\Delta_s\subset\sum_{k\neq j}F_k$ and the pair $(X,F_j+\Delta_s+A)$ is log canonical.\\[2mm]
{\em Step 4.\/}
Assume that the restricted algebra $\res_{F_j}R(X,\mcal C_j\cap\Lambda)$ is finitely generated for every $j$. I will show that
then $\mathfrak R$ is finitely generated.

Let $V=\sum_{j=1}^N\R F_j\simeq\R^N$, and let $\|\cdot\|$ be the Euclidean norm on $V$. By compactness there is a constant $C$ such that every
$\mcal F_j\subset V$ is contained in the closed ball centred at the origin with radius $C$. Let $\deg$ denote the total degree function on $\Lambda$,
i.e.\ $\deg(\sum_{j=1}^N\alpha_jg_j)=\sum_{j=1}^N\alpha_j$; it induces the degree function on elements of $\mathfrak R$.
Let $M$ be a positive integer such that, for each $j$, $\res_{F_j}R(X,\mcal C_j\cap\Lambda)$ is generated by
$\{\sigma_{|F_j}:\sigma\in R(X,\mcal C_j\cap\Lambda),\deg\sigma\leq M\}$, and such that $M\geq CN^{1/2}\max\limits_{i,j}\{\frac{1}{1-\delta_{ij}}\}$.
By H\"{o}lder's inequality we have $\|s\|\geq N^{-1/2}\deg s$ for all $s\in\mcal C\cap\Lambda$, and thus
$$\|s\|/C\geq\max_{i,j}\Big\{\frac{1}{1-\delta_{ij}}\Big\}$$
for all $s\in\mcal C\cap\Lambda$ with $\deg s\geq M$. Let $\mcal H$ be a finite set of generators of the finite dimensional vector space
$$\bigoplus_{s\in\mcal C\cap\Lambda,\deg s\leq M}H^0(X,s)$$
such that for every $j$, the set $\{\sigma_{|F_j}:\sigma\in\mcal H\}$ generates $\res_{F_j}R(X,\mcal C_j\cap\Lambda)$.
I claim that $\mathfrak R$ is generated by $\{\sigma_1,\dots,\sigma_N\}\cup\mcal H$, with $\sigma_j$ as in Step 2.

To that end, take any section $\sigma\in\mathfrak R$ with $\deg\sigma>M$. By definition, possibly by considering monomial parts of $\sigma$ and dividing
$\sigma$ by a suitable product of sections $\sigma_j$, I can assume that $\sigma\in R(X,\mcal C\cap\Lambda)$. Furthermore, by Step 3
there exists $w\in\{1,\dots,N\}$ such that $\sigma\in R(X,\mcal C_w\cap\Lambda)$, thus there is $\tau\in\mcal T\cap\Q^\ell$
such that $\sigma\in H^0(X,r_\sigma\sum_j(f_{\tau j}+b_{\tau j})F_j)$ with $b_{\tau w}=1$. Observe that
$r_\sigma\geq\max\limits_{i,j}\{\frac{1}{1-\delta_{ij}}\}$ since $\|\sum_j(f_{\tau j}+b_{\tau j})F_j\|\leq C$, and in particular
$\frac{r_\sigma-1}{r_\sigma}\geq\delta_{\tau w}$ for every $\tau\in\mcal T$.

Therefore by assumption there are elements $\theta_1,\dots,\theta_z\in\mcal H$ and a polynomial
$\varphi\in\C[X_1,\dots,X_z]$ such that
$\sigma_{|F_w}=\varphi(\theta_{1|F_w},\dots,\theta_{z|F_w})$. Therefore by \eqref{eq:1} in Remark \ref{rem:1},
$$(\sigma-\varphi(\theta_1,\dots,\theta_z))/\sigma_w\in H^0\big(X,r_\sigma\sum\nolimits_j(f_{\tau j}+b_{\tau j})F_j-F_w\big).$$
Since
$$r_\sigma\sum_j(f_{\tau j}+b_{\tau j})F_j-F_w=r_\sigma\Big((f_{\tau w}+{\textstyle\frac{r_\sigma-1}{r_\sigma}})F_w
+\sum_{j\neq w}(f_{\tau j}+b_{\tau j})F_j\Big),$$
we have $r_\sigma\sum_j(f_{\tau j}+b_{\tau j})F_j-F_w\in\mcal C\cap\Lambda$. We finish by descending induction on $\deg\sigma$.\\[2mm]
{\em Step 5.\/}
Therefore it remains to show that for each $j$, the restricted algebra $\res_{F_j}R(X,\mcal C_j\cap\Lambda)$
is finitely generated.

To that end, choose a rational $0<\varepsilon\ll1$ such that $\varepsilon\sum_{k\in I}F_k+A$ is ample for every $I\subset\{1,\dots,N\}$,
and let $A_I\sim_\Q\varepsilon\sum_{k\in I}F_k+A$ be a very general ample $\Q$-divisor.
Fix $j$, and for $I\subset\{1,\dots,N\}\backslash\{j\}$ let
\begin{align*}
\mcal F_{\tau j}^I=(1+f_{\tau j})F_j&+\sum_{k\in I}[1-\varepsilon+f_{\tau k},1+f_{\tau k}]F_k\\
&+\sum_{k\notin I\cup\{j\}}[\delta_{\tau k}+f_{\tau k},1-\varepsilon+f_{\tau k}]F_k.
\end{align*}
Set $\mcal F_j^I=\bigcup_{\tau\in\mcal T}\mcal F_{\tau j}^I$; these are rational polytopes such that
$\mcal F_j=\bigcup_{I\subset\{1,\dots,N\}\backslash\{j\}}\mcal F_j^I$, and therefore $\mcal C_j^I=\R_+\mcal F_j^I$ are rational polyhedral cones
such that $\mcal C_j=\bigcup_{I\subset\{1,\dots,N\}\backslash\{j\}}\mcal C_j^I$.
Furthermore, for every $s\in\mcal C_j^I\cap\Lambda$ we have $s\sim_\Q r_s(K_X+F_j+\Delta_s+A)\sim_\Q r_s(K_X+F_j+\Delta_s'+A_I)$,
where $\Delta_s'=\Delta_s-\varepsilon\sum_{k\in I}F_k\geq0$ and $\lfloor F_j+\Delta_s'+A_I\rfloor=F_j$.

Therefore it is enough to prove that $\res_{F_j}R(X,\mcal C_j^I\cap\Lambda)$ is finitely generated for every $I$.
Fix $I$ and let $h_1,\dots,h_m$ be generators of $\mcal C_j^I\cap\Lambda$. Similarly as in Step 1 of the proof of Theorem \ref{thm:2},
it is enough to prove that the restricted algebra $\res_{F_j}R(X;h_1,\dots,h_m)$ is finitely generated.
For $p$ sufficiently divisible, by the argument above we have $ph_v\sim \rho_v(K_X+F_j+B_v+A_I)=:H_v$, where
$\lceil B_v\rceil\subset\sum_{k\neq j}F_k$, $\lfloor B_v\rfloor=0$, $\rho_v\in\N$ and $A_I$ is a very general ample $\Q$-divisor.
Therefore it is enough to show that $\res_{F_j}R(X;H_1,\dots,H_m)$ is finitely generated by Lemma \ref{lem:1}(1).
But this follows from Theorem \ref{thm:2} and the proof is complete.
\end{proof}

\begin{proof}[Proof of Theorem \ref{cor:can}]
By \cite[Theorem 5.2]{FM00} and by induction on $\dim X$, we may assume $K_X+\Delta$ is big.
Write $K_X+\Delta\sim_\Q B+C$ with $B$ effective and $C$ ample. Let $f\colon Y\rightarrow X$ be a log resolution of $(X,\Delta+B+C)$ and let $H$
be an effective $f$-exceptional divisor such that $f^*C-H$ is ample. Then writing $K_Y+\Gamma=f^*(K_X+\Delta)+E$, where $\Gamma=\B(X,\Delta)_Y$,
we have that $R(Y,K_Y+\Gamma)$ and $R(X,K_X+\Delta)$ have isomorphic truncations. Since $K_Y+\Gamma\sim_\Q(f^*B+H+E)+(f^*C-H)$,
we may assume from the start that $\Supp(\Delta+B+C)$ has simple normal crossings. Let $\varepsilon$ be a small positive rational number and set
$\Delta'=(\Delta+\varepsilon B)+\varepsilon C$. Then $K_X+\Delta'\sim_\Q(\varepsilon+1)(K_X+\Delta)$, and $R(X,K_X+\Delta)$ and
$R(X,K_X+\Delta')$ have isomorphic truncations, so the result follows from Theorem \ref{thm:main}.
\end{proof}

\bibliography{biblio}
\pagestyle{plain}
\end{document}